\documentclass[12pt]{amsart}
\usepackage[T1]{fontenc}
\usepackage[latin1]{inputenc}
\usepackage{typearea}
\usepackage{geometry}
\usepackage{ulem}
\usepackage{amsmath}
\usepackage{amssymb}
\usepackage{latexsym}
\usepackage{enumerate}
\usepackage{verbatim}
\usepackage{amsthm}
\usepackage[all]{xy}
\usepackage{hhline}
\usepackage{epsf} %fuer bildchen
\usepackage{cite}
\usepackage{mathrsfs}
\numberwithin{equation}{section}

\usepackage{color}

\newtheorem{theorem}{{\sc Theorem}}[section]
\newtheorem{cor}[theorem]{{\sc Corollary}}

\newtheorem{lemma}[theorem]{{\sc Lemma}}
\newtheorem{prop}[theorem]{{\sc Proposition}}

\theoremstyle{remark}
\newtheorem{remark}[theorem]{{\sc Remark}}

\theoremstyle{definition}

\newcommand{\R}{\mathbb{R} }
\newcommand{\N}{\mathbb{N} }

\newcommand{\B}{\mathcal{B}}
\newcommand{\F}{\mathcal{F}}
\newcommand{\G}{\mathcal{G}}

\newcommand{\D}{\mathcal{D}}

\newcommand{\Prob}{\mathbb{P}}
\newcommand{\E}{\mathbb{E}}

\newcommand{\Gammabar}{\bar{\Gamma}}

\providecommand{\abs}[1]{\lvert #1\rvert}
\providecommand{\babs}[1]{\bigl\lvert #1\bigr\rvert}
\providecommand{\Babs}[1]{\Bigl\lvert #1\Bigr\rvert}

\providecommand{\norm}[1]{\lVert #1\rVert}
\providecommand{\fnorm}[1]{\lVert #1\rVert_\infty}

\DeclareMathOperator{\Var}{Var}

\DeclareMathOperator{\dom}{dom}
\DeclareMathOperator{\Cov}{Cov}

\renewcommand{\phi}{\varphi}
\renewcommand{\epsilon}{\varepsilon}

\renewcommand{\rho}{\varrho}

\begin{document}
{
\title[The Gamma Stein equation]{The Gamma Stein equation  \\  and non-central de Jong theorems \\~\\\small{\it Dedicated to the memory of Charles M. Stein }  }

\author{  Christian D\"obler \and Giovanni Peccati}
\thanks{\noindent Universit\'{e} du Luxembourg, Unit\'{e} de Recherche en Math\'{e}matiques \\
E-mails: christian.doebler@uni.lu, giovanni.peccati@gmail.com }
\begin{abstract} We study the Stein equation associated with the one-dimensional Gamma distribution, and provide novel bounds, allowing one to effectively deal with test functions supported by the whole real line. We apply our estimates to derive new quantitative results involving random variables that are non-linear functionals of random fields, namely: (i) a non-central quantitative de Jong theorem for sequences of degenerate $U$-statistics satisfying minimal uniform integrability conditions, significantly extending previous findings by de Jong (1990), Nourdin, Peccati and Reinert (2010) and D\"obler and Peccati (2016), (ii) a new Gamma approximation bound on the Poisson space, refining previous estimates by Peccati and Th\"ale (2013), and (iii) new Gamma bounds on a Gaussian space, strengthening estimates by Nourdin and Peccati (2009). As a by-product of our analysis, we also deduce a new inequality for Gamma approximations {\it via} exchangeable pairs, that is of independent interest.
\end{abstract}

\maketitle

\section{Introduction}\label{intro}

\subsection{Overview}

The aim of this paper is to derive new explicit estimates for one-dimensional Gamma approximations, and then to apply our general findings to derive several non-central approximation results for sequences of random variables that have the form of non-linear functionals of a random measure. The random measures we are interested in are either the empirical measure associated with a sequence of independent random variables, or a Poisson or Gaussian measure. As discussed below, our applications significantly refine and generalise previous results about the Gamma approximation of degenerate and not necessarily symmetric $U$-statistics \cite{DP16, deJo87, deJo89, deJo90, PTh}, of smooth random variables on the Poisson space \cite{PSTU, PTh}, and of smooth functionals of a Gaussian field \cite{NPR-aop,NP09a,NP-ptrf}.

\medskip

From now on, for fixed $r,\lambda\in (0,\infty)$, we will denote by $\Gamma(r,\lambda)$ the \textit{Gamma distribution} with \textit{shape parameter} $r$ and \textit{rate} $\lambda$ which has probability density function (p.d.f.)
\begin{equation*}
p_{r,\lambda}(x)=\begin{cases}
                  \frac{\lambda^r}{\Gamma(r)}x^{r-1}e^{-\lambda x}\,,&\text{if }x>0\\
                  0\,,&\text{otherwise,}
                 \end{cases}
\end{equation*}
where 
\begin{equation*}
 \Gamma(t):=\int_0^\infty x^{t-1}e^{-x}dx
\end{equation*}
denotes the \textit{Euler Gamma function}. We denote the corresponding distribution function by $F_{r,\lambda}$. It is well-known that $X_{r,\lambda}\sim\Gamma(r,\lambda)$ has mean $r/\lambda$ and variance $r/\lambda^2$ and that, if $Y=aX_{r,\lambda}$ for some $a>0$, then $Y$ has distribution $\Gamma(r,a^{-1}\lambda)$.
For $\nu>0$, we also denote by $\Gammabar(\nu)$ the so-called \textit{centered Gamma distribution} with parameter $\nu$ which by definition is the distribution of 
\[Z_\nu:=2X_{\nu/2,1}-\nu\,,\]
where, again, $X_{\nu/2,1}$ has distribution $\Gamma(\nu/2,1)$. Notice that, if $\nu$ is an integer, then $\Gammabar(\nu)$ has a centered $\chi^2$ distribution with $\nu$ degrees of freedom. According to the previous discussion, one has that
\begin{equation*}
 \E[Z_\nu]=0\quad\text{and}\quad \Var(Z_\nu)=\E[Z_\nu^2]=2\nu\,;
\end{equation*}
also, the following moment identity (already exploited in \cite{NP09a}), will play an important role throughout the paper:
\begin{equation}\label{e:lcm}
 \E[Z_\nu^4]-12\E[Z_\nu^3]-12\nu^2+48\nu=0\,.
\end{equation}

\medskip

One of our principal aims in the sections to follow is to obtain several explicit estimates on quantities of the type
$$
d(W,X_{r,\lambda}) := \sup_{h\in \mathcal{H}} \left| \E[h(W)] - \E[h(X_{r,\lambda})]\right|,
$$
where $\mathcal{H}$ is a suitable class of test functions. The strategy we will adopt in order to do so, is to derive new estimates on the solutions of the {\it Gamma Stein equation}
\begin{equation}\label{steineq1}
 xf'(x)+(r-\lambda x)f(x)=h(x)-\E[h(X_{r,\lambda})], \quad x\in \R,
\end{equation}
where $h$ is an element of $\mathcal{H}$, and then to effectively use our bounds in the framework of {\it exchangeable pairs} (see \cite{DP16, St86}). We will see that our results significantly extend the classical findings by \cite{Luk} and Pickett \cite{Pick}, as well as the recent estimates from \cite{GRP}. In particular, one crucial feature of our approach is that we will be able to directly study the Stein equation \eqref{steineq1} on the whole real line, although the target distribution $\Gamma(r,\lambda)$ is supported on the positive real axis. As discussed in Section \eqref{ss:gaussintro}, in the specific case of Gamma approximations on a Gaussian space, our results remarkably allow one to obtain quantitative limit theorems in the 1-Wasserstein distance (see below for definitions).

\medskip

As anticipated, our main motivation comes from the study of the non-central fluctuations of random objects which can be expressed in terms of iterated stochastic integrals with respect to a given random measure. The next three subsections contain a detailed discussion of our main applications to degenerate $U$-statistics and multiple integrals on the Poisson and Gaussian spaces.}

\subsection{A non-central de Jong theorem}\label{ss:dejongintro}
Let $X_1,\dotsc,X_n$ be independent random variables on some generic probability space $(\Omega,\F,\Prob)$ and with values in arbitrary measurable spaces $(E_1,\mathcal{E}_1),\dotsc,(E_n,\mathcal{E}_n)$.
In the recent paper \cite{DP16} we were able to prove error bounds for the uni- and multivariate normal approximation of (vectors of) degenerate, non-symmetric $U$-statistics of the data vector $X=(X_1,\dotsc,X_n)$. 
In particular, we were able to provide a complete quantitative extension of a CLT by de Jong \cite{deJo90} which roughly states that a normalized sequence $W_n$, $n\in\N$, of such $U$-statistics converges weakly 
to the standard normal disribution if the sequence of fourth {moments} converges to $3$ and some asymptotic {Lindberg-type} condition is satisfied  --- see formula \eqref{dpbound} below.

\smallskip

The main abstract results of the present paper are used to continue such a line of research by dealing with the approximation of such a degenerate, non-symmetric $U$-statistic 
%of the data vector $X=(X_1,\dotsc,X_n)$
by a centered Gamma distribution. More precisely, assume that  
\begin{equation*}
 \psi:\prod_{j=1}^n E_j\rightarrow\R\quad\text{is}\quad\bigotimes_{j=1}^n\mathcal{E}_j-\B(\R)\text{ - measurable}
\end{equation*}
and that 
\begin{equation*}
W:=\psi(X_1,\dotsc,X_n) \in L^4(\Prob)
\end{equation*}
satisfies 
\begin{equation}\label{stw}
 \E[W]=0\quad\text{and}\quad \E[W^2]=2\nu
\end{equation}
for some $\nu>0$.
We write 
\begin{equation*}
 [n]:=\{1,\dotsc,n\}
\end{equation*}
and for $J\subseteq[n]$ we define
\begin{equation*}
 \F_J:=\sigma(X_j,\,j\in J)\,.
\end{equation*}
We denote by 
\begin{equation}\label{hoeffding}
 W=\sum_{J\subseteq[n]}W_J
\end{equation}
the {\it Hoeffding decomposition} of $W$ {(see e.g. \cite{DP16, Hoeffding, KR, KB-book, ser-book, vitale})}. Note that this means that, for each $J\subseteq[n]$,  $W_J$ is $\F_J$-measurable and that 
\begin{equation*}
 \E[W_J\,|\,\F_K]=0\,,
\end{equation*}
whenever $J\nsubseteq K$. It is well-known that $W$ admits a Hoeffding decomposition of the type \eqref{hoeffding}, as long as $W\in L^1(\Prob)$ and that 
it is almost surely unique and given by 
\begin{equation}\label{hdw}
W_J=\sum_{L\subseteq J}(-1)^{\abs{J}-\abs{L}}\E\bigl[W\,\bigl|\,\F_L\bigr]\,,\quad{J\subseteq[n]}\,.
\end{equation}
We can thus write 
\begin{equation*}
 W_J=\psi_J(X_j,\,j\in J)
\end{equation*}
for some measurable function
\begin{equation*}
 \psi_J:\prod_{j\in J}E_j\rightarrow\R\,,\quad J\subseteq[n]\,.
\end{equation*}
Let us also define
\begin{equation*}
 \sigma_J^2:=\Var(W_J)\,,\quad J\subseteq[n]\,.
\end{equation*}
One major assumption in what follows will be that, for some fixed integer $d\in[n]$, $W$ is a \textit{ degenerate $U$-statistic of order $d$} (or $d$-degenerate $U$-statistic), i.e. that the 
Hoeffding decomposition \eqref{hoeffding} has the form 
\begin{equation}\label{dhom}
 W=\sum_{J\in\D_d}W_J\,,
\end{equation}
where 
\[\D_d:=\{J\subseteq[n]\,:\,\abs{J}=d\}\]
denotes the collection of all $\binom{n}{d}$ $d$-subsets of $[n]$, {i.e., we assume that $W_K=0$ $\Prob$-a.s. whenever $\abs{K}\not=d$.} Hence, we have 
\begin{equation}\label{nonsymU}
 W=\psi(X_1,\dotsc,X_n)=\sum_{J\in\D_d} \psi_J(X_j,\,j\in J)\,.
 \end{equation}
Furthermore, we define the quantities
\begin{equation*}
 \rho^2:=\rho_n^2:=\max_{1\leq i\leq n}\sum_{\substack{K\in\D_d:\\i\in K}}\sigma_K^2\quad\text{and}\quad D:=D_n:=\max_{J\in\D_d}\frac{\E\bigl[W_J^4\bigr]}{\sigma_J^4}\,.
\end{equation*} 
 
\medskip 

One of the main results of the present paper is an explicit upper bound on a certain probability distance between the law of $W$ and $\Gammabar(\nu)$. For $k\in\N$, denote by $\mathcal{H}_k$ the class of those $(k-1)$-times differentiable test functions 
$h$ on $\R$ such that $h^{(k-1)}$ is Lipschitz-continuous and we have   
\begin{equation*}
 \fnorm{h^{(l)}}\leq 1 \quad\text{for}\quad l=1,\dotsc,k\,.
\end{equation*}
For real random variavbles $X$ and $Y$ such that $\E\abs{X},\E\abs{Y}<\infty$ we denote by 
\begin{equation*}
 d_k(X,Y):=d_k\bigl(\mathcal{L}(X),\mathcal{L}(Y)\bigr):=\sup_{h\in\mathcal{H}_k}\babs{\E[h(X)]-\E[h(Y)]}
\end{equation*}
the distance {between the distributions of $X$ and $Y$} induced by the class $\mathcal{H}_k$; observe that $d_1$ coincides with the classical 1-{\it Wasserstein distance}, see e.g. \cite[Appendix C]{NouPecbook} and the references therein.
The next theorem estimates the $d_2$-distance between the law of $W$ and $\Gammabar(\nu)$ in terms of the analogous linear combination of the moments of $W$ as well as in terms of the quantities $\rho_n^2$ and $D_n$. 

\begin{theorem}\label{maintheo}
 Under the above assumptions we have the bound 
{\begin{align*}
d_2(W,Z_\nu)&\leq\frac{\max\bigl(1,\frac{2}{\nu}\bigr)}{\sqrt{3}}\sqrt{\babs{\E[W^4]-12\E[W^3]-12\nu^2+48\nu}}\\
&\;+\frac{\bigl(2\sqrt{3}+4\sqrt{\nu}\bigr)\max\bigl(1,\frac{2}{\nu}\bigr)+4\sqrt{\nu} }{3\sqrt{d}}\sqrt{C_dD_n\rho_n^2}\,,
\end{align*}}
{where} $C_d$ is a finite constant which only depends on $d$.
\end{theorem}

One should immediately notice that the factors
$$
\frac{\max\bigl(1,\frac{2}{\nu}\bigr)}{\sqrt{3}}, \quad \mbox{and} \quad \frac{\bigl(2\sqrt{3}+4\sqrt{\nu}\bigr)\max\bigl(1,\frac{2}{\nu}\bigr)+4\sqrt{\nu} }{3\sqrt{d}},
$$
both diverge to infinity as $\nu \to 0$. As formally discussed in Remark \ref{r:explosion}, this somewhat undesirable feature seems to be unavoidable: in particular, such a phenomenon is related to the fact that, for our applications, we need to be able to deal with random variables whose distribution is possibly supported by the whole real line.

\smallskip 

The estimate in Theorem \ref{maintheo} immediately yields the following limit result.

\begin{cor}\label{mtcor}
 Fix $\nu>0$ and an integer $d\geq 1$ and let $\{n_m : m\geq 1\}$ be a sequence of integers diverging to infinity. Let $\{W_m : m\geq 1 \}$ be a sequence of centered, degenerate $U$-statistics of order $d$ with $\E[W_m^2]=2\nu$, 
 such that each $W_m$ is a function of the vector of independent variables $(X_1^{(m)},...,X_{n_m}^{(m)})$. Then, if 
 \[\lim_{m\to\infty} \bigl(\E[W_m^4]-12\E[W_m^3]-12\nu^2+48\nu\bigr)=0=\lim_{m\to\infty} D_{n_m}\rho^2_{n_m}\,,\]
 the sequence $\{W_m : m\geq 1 \}$ converges in distribution to $Z_\nu$.
\end{cor}

Plainly, the asymptotic relation $\lim_{m\to\infty} D_{n_m}\rho^2_{n_m}= 0$ is verified whenever the sequence $\{D_{n_m}\}$ is bounded, and $\rho^2_{n_m}\to 0$; see the discussion below. 

\smallskip

It is also instructive to compare Theorem \ref{maintheo} and Corollary \ref{mtcor} with the main findings of \cite{DP16}, applying to the case where the assumption $\E[W^2]=2\nu$ in \eqref{stw} is replaced by $\E[W^2]=1$. In this framework, letting $Z$ be a standard normal random variable, one deduces from \cite[Theorem 1.3]{DP16} that 
\begin{align}\label{dpbound}
 d_1(W,Z) &\leq \Bigl(\sqrt{\frac{2}{\pi}}+\frac{4}{3}\Bigr)\sqrt{\babs{\E[W^4]-3}}+\sqrt{\kappa_d}\Bigl(\sqrt{\frac{2}{\pi}}+ \frac{2\sqrt{2}}{\sqrt{3}}\Bigr)\rho_n\,,
\end{align}
where $\kappa_d$ is a finite constant which only depends on $d$. As demonstrated in \cite{DP16}, from \eqref{dpbound} one can immediately deduce de Jong's theorem \cite{deJo90}: \textit{Fix $d\geq 1$, and let $\{n_m : m\geq 1\}$ be a sequence of integers diverging to infinity. Let $\{W_m : m\geq 1 \}$ be a sequence of unit variance degenerate $U$-statistics of order $d$, such that each $W_m$ is a function of the vector of independent variables $(X_1^{(m)},...,X_{n_m}^{(m)})$. %, and the parameter $\rho^2(m)$ is defined in the obvious way. 
Then, as $m\to \infty$, if $\E[ W_m^4]\to 3$ and $\rho_{n_m}^2\to 0$, one has that $W_m$ converges in distribution towards a standard Gaussian random variable.}

\begin{remark}\label{mtrem}
\begin{enumerate}[(a)]
\item {Thanks to relation \eqref{e:lcm}}, Corollary \ref{mtcor} is an analog of de Jong's theorem \cite{deJo90} in the context of a Gamma limit. 
{ \item { As discussed in Section \ref{ss:approach} below, we believe that, in view of fundamental structural results from \cite{EdVi15}, the bound appearing in Theorem \ref{maintheo} is the best de Jong-type estimate on the Gamma approximation of $U$-statistics that can be achieved by using Stein's method.} Using the statement of Lemma \ref{smoothing} below, one can also immediately deduce a bound { (with completely explicit constants)} on the Wasserstein distance between $W$ and $Z_\nu$ whose order is the square root of the rate of convergence we get for the $d_2$-distance. 
{ We also observe that, by applying} techniques similar to those used in the proof of Lemma 2.3 in \cite{FulRos13} we could obtain a bound on the Kolmogorov distance whose order would be power $1/3$ of the rate for the $d_2$-distance, at least in the case $\nu\geq 2$, { that is}, when $Z_\nu$ has a bounded density. We omit the details of this computation and refer to \cite{FulRos13} for further information. }
\item We conjecture that, analogously to the bounds on normal approximations derived in \cite{DP16}, the quantity $D_n$ could be removed from the bound in Theorem \ref{maintheo} and, hence, also from the limit theorem stated in Corollary \ref{mtcor}. 
\item In \cite{NPR-aop} the authors prove an error bound on the centered Gamma approximation (for integer $\nu$) of homogeneous multilinear forms in independent and normalized real-valued random variables $(X_i)_{i\in\N}$. 
These form a particularly important example class of degenerate, non-symmetric $U$-statistics. Their bound also involves the quantities $\babs{\E[W^4]-12\E[W^3]-12\nu^2+48\nu}$, $\rho_n^2$ and $\beta:=\sup_{i\in\N}\E[X_i^4]$ and it is easy to see that the condition $\beta<\infty$ is in fact equivalent to the condition $\sup_{n\in\N}D_n<\infty$ in this special situation. 
Thus, Theorem \ref{maintheo} and Corollary \ref{mtcor} can be seen as an extension and improvement of the bounds and limit theorems from \cite{NPR-aop} to a wider class of statistics. 
\end{enumerate}
\end{remark}

The following new result gives a universal bound for the Wasserstein distance in terms of the $d_2$-distance. The proof is deferred to Section \ref{appendix}.

{ 
\begin{lemma}\label{smoothing}
Let $X$ and $Y$ be any real-valued random variables with $\E\abs{X}<\infty$ and $\E\abs{Y}<\infty$. Then, we have the bound
\begin{equation*}
d_1(X,Y)\leq \frac{4}{\sqrt{\pi}}\sqrt{d_2(X,Y)}\,, 
\end{equation*}
whenever $d_2(X,Y)\leq1$.
\end{lemma}
}

\subsection{Gamma limits on the Poisson space}\label{ss:poissintro} 
In this subsection, we describe how our new bounds on the solution to the Gamma Stein equation \eqref{steineq1}, yield new analytic estimates for the Gamma approximation of functionals of a Poisson random measure. We will first briefly introduce the setup and some necessary notation. Further technical details are provided in Section \ref{poisson}. For any unexplained notions we refer to the recent book \cite{PecRei16}, in particular Chapter 1 \cite{Lastsv}, as well as to the existing related literature, e.g. \cite{PSTU, LRP1, LRP2, PTh}. We stress that limit theorems and probabilistic approximations involving non-linear functionals of a Poisson measure have gained enormous momentum in recent years, specially in connections with the large scale analysis of random geometric structures -- see again \cite{PecRei16}, and the references therein.

\smallskip

We now fix a Polish space $\mathcal{Z}$ as well as a $\sigma$-finite measure $\mu$ on the Borel-$\sigma$-field $\mathscr{Z}$ on $\mathcal{Z}$. Furthermore, we let 
\[\mathscr{Z}_\mu:=\{B\in\mathscr{Z}\,:\,\mu(B)<\infty\}\]
and denote by 
\begin{equation*}
\eta=\{\eta(B)\,:\,B\in\mathscr{Z}_\mu\}
\end{equation*}
a \textit{Poisson measure} on $(\mathcal{Z},\mathscr{Z})$ with \textit{control} $\mu$, defined on a suitable probability space $(\Omega,\F,\Prob)$. We recall that the distribution of $\eta$ is completely determined by the following two facts: (i) { for each finite sequence $B_1,\dotsc,B_m$ of disjoint sets in $\mathscr{Z}_\mu$, the random variables 
$\eta(B_1),\dotsc,\eta(B_m)$ are independent,} and (ii) that for every $B\in\mathscr{Z}_\mu$, the random variable $\eta(B)$ has the Poisson distribution with mean $\mu(B)$. For $B\in\mathscr{Z}_\mu$, we also write 
$\hat{\eta}(B):=\eta(B)-\mu(B)$ and denote by 
\[\hat{\eta}=\{\hat{\eta}(B)\,:\,B\in\mathscr{Z}_\mu\}\]
the \textit{compensated Poisson measure} associated with $\eta$. Without loss of generality, we may and will assume that $\F=\sigma(\eta)$. 

\smallskip

Our main result in this section involves the following Malliavin operators: (i) the {\it Malliavin derivative} $D$, (ii) the {\it generator of the Ornstein-Uhlenbeck semigroup} $L$, and (iii) the {\it pseudo-inverse} of $L$, written $L^{-1}$. Formal definitions and details are provided in Section \ref{poisson}. Here, we only recall that the spectrum of $L$ is given by the negative integers $\{-p : p=0,1,2,\ldots\}$ and that $F\in {\rm Ker}(L+pI)$ (that is, $F$ is an eigenfunction of $L$, with eigenvalue $-p$) if and only if $F = I_p(f)$, where $I_p$ indicates a multiple Wiener-It\^o integral of order $p$ with respect to $\widehat{\eta}$, and $f$ is a suitable square-integrable kernel. The eigenspace ${\rm Ker}(L+pI)$ is customarily called the $p$th {\it Wiener chaos} associated with $\eta$.

\smallskip

The next statement -- whose proof exploits our new results on the solution to the Stein equation \eqref{steineq1} -- is our main estimate on the Poisson space: in particular, its second part contains the announced result for multiple Wiener-It\^o integrals. Proofs are deferred to Section \ref{poisson}.

\begin{theorem}\label{poitheo}
 Let $F\in L^2(\Prob)$ be centered, and assume that $F$ belongs to the domain of the Malliavin derivative operator $D$. Then, we have the bounds
 \begin{align}
  d_2(F,Z_\nu)&\leq\max\Bigl(1,\frac{2}{\nu}\Bigr)\E\Babs{2(F+\nu)-\langle DF,-DL^{-1}F\rangle _{L^2(\mu)}}\notag\\
  &\hspace{2cm}+\max\Bigl(1,\frac{1}{\nu}+\frac12\Bigr)\int_{\mathcal{Z}}\E\bigl[\abs{D_zF}^2\abs{D_zL^{-1}F}\bigr]\mu(dz)\label{genpb1}\\
  &\leq \max\Bigl(1,\frac{2}{\nu}\Bigr)\sqrt{\E\Bigl[\bigl(2(F+\nu)-\langle DF,-DL^{-1}F\rangle _{L^2(\mu)}\bigr)^2\Bigr]} \notag\\
  &\hspace{2cm}+\max\Bigl(1,\frac{1}{\nu}+\frac12\Bigr)\int_{\mathcal{Z}}\E\bigl[\abs{D_zF}^2\abs{D_zL^{-1}F}\bigr]\mu(dz)\label{genpb2}\,.
 \end{align}
Here, we have used the standard notation
\begin{equation*}
 \langle DF,-DL^{-1}F\rangle_{L^2(\mu)}=-\int_{\mathcal{Z}}\bigl(D_zF\bigr)\bigl(D_zL^{-1}F\bigr)\mu(dz)\,.
\end{equation*}
If, furthermore, $F=I_p(f)$ for some $p\geq 1$ and some square-integrable kernel $f$, then 
\begin{align*}
 \langle DF,-DL^{-1}F\rangle_{L^2(\mu)}&=p^{-1}\norm{DF}^2_{L^2(\mu)}\quad\text{and}\\
 \int_{\mathcal{Z}}\E\bigl[\abs{D_zF}^2\abs{D_zL^{-1}F}\bigr]\mu(dz)&=p^{-1}\int_{\mathcal{Z}}\E\bigl[\abs{D_zF}^3\bigr]\mu(dz)
\end{align*}
so that the previous estimates becomes
\begin{align}
 d_2(F,Z_\nu)&\leq\max\Bigl(1,\frac{2}{\nu}\Bigr)\E\Babs{2(F+\nu)-p^{-1}\norm{DF}^2_{L^2(\mu)}} \notag\\
 &\;+p^{-1}\max\Bigl(1,\frac{1}{\nu}+\frac{1}{2}\Bigr)\int_{\mathcal{Z}}\E\bigl[\abs{D_zF}^3\bigr]\mu(dz)\label{spb1} \\
 &\leq \max\Bigl(1,\frac{2}{\nu}\Bigr)\sqrt{\E\Bigl[\bigl(2(F+\nu)-p^{-1}\norm{DF}^2_{L^2(\mu)}\bigr)^2\Bigr]}\notag\\
&\; +p^{-1}\max\Bigl(1,\frac{1}{\nu}+\frac{1}{2}\Bigr)\int_{\mathcal{Z}}\E\bigl[\abs{D_zF}^3\bigr]\mu(dz)\label{spb2}\,.
\end{align}
\end{theorem}

\smallskip

\begin{remark}{\rm

The content of Theorem \ref{poitheo} should be directly compared with \cite[Theorem 2.1]{PTh}, according to which
\begin{align*}
  d_3(F,Z_\nu)&\leq c_1\E\Babs{2(F+\nu)_+ -\langle DF,-DL^{-1}F\rangle _{L^2(\mu)}} \\
  &\hspace{2cm}+c_2\int_{\mathcal{Z}}\E\bigl[\abs{D_zF}^2\abs{D_zL^{-1}F}\bigr]\mu(dz) \\
  &\hspace{2cm}+2c_1\int_{\mathcal{Z}} \E[(D_zF {\bf 1}_{( F>-\nu)}) (D_zF) \abs{D_zL^{-1}F}]\mu(dz),
 \end{align*}
 where $c_1, c_2$ are explicit constants uniquely depending on $\nu$. Note that our estimate \eqref{genpb1} improves on such an estimate in three ways: (i) the distance $d_3$ is replaced by the less smooth distance $d_2$, (ii) the first expectation on the right-hand side does not involve the positive part of $F+\nu$, and (iii)  the third term in the bound has been completely removed. As will become evident in the proof, Points (i) and (iii) are a direct consequence of the fact that our approach allows us to solve and control equation \eqref{steineq1} on the whole real line, thus obtaining more tractable solutions than those used in \cite{PTh}. Note that our bound can be directly used to deduce simplified proofs of the other estimates proved in \cite{PTh}, like e.g. \cite[Theorem 2.6 and Proposition 2.9]{PTh}}. Details are left to the reader.
\end{remark}

\subsection{Gamma limits on a Gaussian space}\label{ss:gaussintro} 

We conclude this section by showing how the results of the present paper can also be used to give better estimates on the Gamma approximation of non-linear functionals of Gaussian fields, thus improving results from \cite{NP-ptrf, NPo-spa}. For the sake of conciseness, in this section we will keep explicit definitions to a minimum, and refer the reader to the monograph \cite{NouPecbook} for any unexplained notion or detail. 

\smallskip

Now let $\mathscr{H}$ be a real separable Hilbert space, and let $X = \{X(h) : h\in \mathscr{H}\}$ be an {\it isonormal Gaussian process} over $\mathscr{H}$. We assume that $X$ is defined on a suitable probability space $(\Omega, \F, \Prob)$, and that $\F = \sigma (X)$. Similarly to the previous section, we associate to $X$ the following canonical Malliavin operators: (i) the {\it Malliavin derivative} $D$ (whose domain is indicated by $\mathbb{D}^{1,2}$), (ii) the {\it generator of the Ornstein-Uhlenbeck semigroup} $L$, and (iii) the {\it pseudo-inverse} of $L$, written again $L^{-1}$. As on the Poisson space, the spectrum of $L$ is given by the negative integers $\{-p : p=0,1,2,\ldots\}$ and one has that $F\in {\rm Ker}(L+pI)$ (that is, the $p$th {\it Wiener chaos} of $X$) if and only if $F = I_p(f)$, where $I_p$ indicates a multiple Wiener-It\^o integral of order $p$, and $f$ is an element of the symmetric tensor product $\mathscr{H}^{\odot p}$. 

\smallskip

One has the following estimate (recall that $d_1$ corresponds to the 1-Wasserstein distance).

\begin{theorem}\label{t:gaussgamma} Let $F$ be centered element of $\mathbb{D}^{1,2}$ and fix $\nu>0$. Then, 
\begin{eqnarray}\label{e:gaussgamma}
d_1(F, Z_\nu) &\leq& \max\Bigl(1,\frac{2}{\nu}\Bigr)\E  \Babs{ \, \E \big\{ 2(F+\nu)-\langle DF,-DL^{-1}F\rangle _{\mathscr{H} }  \,  | \, F \big \} \, } \\
&\leq& \max\Bigl(1,\frac{2}{\nu}\Bigr)\E\left [ \left( 2(F+\nu)-\langle DF,-DL^{-1}F\rangle _{\mathscr{H} } \right)^2\right]^{1/2}.\notag
\end{eqnarray}
If $F\in {\rm Ker}(L+pI)$ for some integer $p\geq 2$, then the previous estimate becomes
\begin{eqnarray}\label{e:gaussgamma2}
d_1(F, Z_\nu) &\leq& \max\Bigl(1,\frac{2}{\nu}\Bigr)  \E\Babs{ \E \big\{2(F+\nu)- p^{-1} \| DF\|^2_{\mathscr{H}}  \,  | \, F \big \}  }.
\end{eqnarray}
\end{theorem}

Inequality \eqref{e:gaussgamma} improves \cite[Theorem 3.11]{NP-ptrf}, where a similar upper bound is proved for a smoother distance (written $d_{\mathscr{H}_2}$ therein) involving test functions of class $C^2$ with bounded derivatives. By inspection of the proofs contained in \cite{NP-ptrf}, one sees that such a smoothness requirement on test functions is indeed an artefact of the bounds contained in \cite{Luk}. By combining Theorem \ref{t:gaussgamma} with the main findings from \cite{NPo-spa} and with some computations from \cite{APP}, one also obtains the following non-trivial quantitative characterisation of Gamma convergence in total variation inside a fixed sum of Wiener chaoses. We recall that, given two real-valued random variables $X,Y$, the {\it total variation distance} between the distributions of $X$ and $Y$ is given by
$$
d_{TV}(X,Y) = \sup_{A\in \mathscr{B}(\R)} \left | \Prob[X\in A] - \Prob[Y\in A]\right|,
$$
where $\mathscr{B}(\R)$ stands for the class of all Borel subsets of $\R$.

\begin{prop}\label{p:gaussgamma} Fix $\nu>0 $, as well as an integer $m\geq 2$, and let  $\{F_n : n\geq 1\} \subseteq \bigoplus_{p=1}^m{\rm Ker}(L+pI)$ be such that $\E[F_n^2]\to 2\nu$.  Then, $F_n$ converges in distribution to $Z_\nu$ if and only if
\begin{equation}\label{e:sar}
\E\Babs{ \E \big\{ 2(F_n+\nu)-\langle DF_n,-DL^{-1}F_n\rangle _{\mathscr{H} }  \,  | \, F_n \big \} }\to 0, \quad n\to \infty,
\end{equation}
and there exists a finite constant $c>0$ (not depending on $n$)
such that
\begin{equation}\label{e:sar1}
d_{TV}(F_n, Z_\nu) \leq c \left (  \E\Babs{ \E \big\{ 2(F_n+\nu)-\langle DF_n,-DL^{-1}F_n\rangle _{\mathscr{H} }  \,  | \, F_n \big \} }\right)^{1/{2m+1} }.
\end{equation}
\end{prop}
 One has also to observe that, according to \cite{NP09a}, if the sequence $\{F_n\}$ in Proposition \ref{p:gaussgamma}  is such that $\{F_n\}\subseteq {\rm Ker}(L+mI)$ and \eqref{e:sar} is verified, then necessarily $m$ is an even integer. See also \cite{AS, KuTu} for some related limit theorems. The proofs of Theorem \ref{t:gaussgamma} and Proposition \ref{p:gaussgamma} are given in Section \ref{gauss}.

{
\subsection{About our approach and assumptions}\label{ss:approach}

{ We will now} make some technical remarks about { the methods and assumptions adopted in the present paper.} 

\begin{itemize}

\item[(i)] First of all, we recall that { the Stein equation associated with a} given distribution is in general not unique, and several approaches are available in order to { select} a specific one. One of these methods, the so-called \textit{density approach} (see e.g. \cite{CGS} and \cite{LRS17}), suggests a Stein equation of the form 
\begin{equation*}
 f'(x)-\psi(x) f(x)=h(x)-\E[h(X_{r,\lambda})]\,,
 \end{equation*}
where $\psi(x):=\frac{d}{dx}\log p_{r,\lambda}(x)$, $x>0$, is the log-derivative of the density function. It is easy to see that, here, $\psi(x)=\frac{r-1-\lambda x}{x}$ is a { genuinely} rational function of $x$ (unless $r=1$), which makes this equation very difficult to apply in concrete situations { involving} probability approximations. Note in particular that, in the three examples presented in Subsections \ref{ss:dejongintro}--\ref{ss:gaussintro}, it is for us of fundamental importance to have a linear coefficient of $f(x)$ in the Stein equation, which makes { the choice of} \eqref{steineq1} inevitable. 

\item[(ii)] Due to their wide applicability (in particular in combination with Malliavin calculus techniques) the class of first order Stein equations having a linear coefficient of $f$, characterizing some 
absolutely continuous distribution $\mu$ on $\R$, has been well-studied in the recent literature (see e.g. \cite{KuTu12}, \cite{EdVi15}, \cite{DoeBeta}). In particular, in \cite{EdVi15} (see Remark 10 and the last paragraph on page 200) the authors prove by means of a universal counterexample the following remarkable fact: {\it { if the support of the distribution $\mu$ is a strict sub-interval of $\R$ and if $\mu$ is characterized by a Stein equation of the type 
\[a(x) f'(x)+(b-cx)f(x)=h(x)-\int hd\mu\]
for $b,c$ real constants ($c\not=0$), then there is no} finite constant $M>0$ such that $\fnorm{f_h''}\leq M\fnorm{h'}$ holds for all Lipschitz-continuous functions $h$ on $\R$.}  Here, $f_h$ denotes the usual 
solution of the Stein equation. { To the best of our expertise, this fact shows that the bounds on the solution $f_h$ of \eqref{steineq1} presented in Theorem \ref{bounds1} are the best that the technology of Stein's method can presently achieve for the Gamma distribution.} Such a structural result also immediately entails that, { as far as the Gamma distribution is concerned}, one necessarily has to assume more smoothness on the test function $h$, in order to be able to work with second derivatives of the Stein's solution. We observe that, except on the Gaussian space, where bounds on the first derivative $f'_h$ are sufficient due to the diffusiveness of the involved Ornstein-Uhlenbeck operator $L$, in the more general framework of Subsections \ref{ss:dejongintro} and \ref{ss:poissintro} one necessarily has to work with second derivatives as well, { because of the intrinsic discrete nature of the considered objects. }

\item[(iii)] If one aims at less smooth distances -- like the prominent Kolmogorov or Wasserstein distances -- then one needs to implement some additional smoothing procedure. { As it is typical, this inevitably} comes at the price { of a worse rate of convergence}. The new Lemma \ref{smoothing} stated above provides such a smoothing result which, roughly speaking, bounds the Wasserstein distance of quite arbitrary distributions in terms of a distance induced by test functions which have one additional order of smoothness. 

\item[(iv)] We stress that, with the exception of the references \cite{Doe12a} and \cite{Doe12c}, none of the references mentioned so far consider the Stein equation beyond the support interval of the corresponding distribution. 
For certain applications this is indeed not necessary, because, by applying some truncation procedure, one can force any random variable to have support in a given interval. However, for all three applications considered in this paper, applying 
truncation would immediately destroy the most important structural property of the random variables under consideration: { In Subsection \ref{ss:dejongintro}, the truncated random variable would no longer be a degenerate $U$-statistic of 
a given order and one would therefore have to work with a full Hoeffding decomposition; similarly, in the situations dealt with in Subsections \ref{ss:poissintro}--\ref{ss:gaussintro} the chaotic decomposition of the truncated random variable would immedieately be infinite and, thus, not directly amenable to computations.} Since in general our random variables may have support equal to the whole real line, it is for us imperative to deal with the Stein equation \eqref{steineq1} and its solution also outside the support of the target distribution.  

\item[(v)] Our main applications, see Theorem \ref{maintheo} and Corollary \ref{mtcor}, concern the centered Gamma approximation of a degenerate, not necessarily symmetric $U$-statistic $W$ of order $d$, { based on} some independent random sample $X_1,\dotsc,X_n$ ($n\geq d$) --- see Subsection \ref{ss:dejongintro}. Here, we would like to stress that the classical results about the asymptotic distributions of $U$-statistics obtained in e.g. \cite{DynMan83} and \cite{RubVi80} (see also \cite{ser-book} and \cite{Greg77} for the case $d=2$) do not apply. First of all, our data random variables $X_1,\dotsc,X_n$ are not necessarily i.i.d.. Moreover, and even more importantly, our $U$-statistics are in general {\it non symmetric}, and have kernels that in general {\it depend on} $n$ (see equation \eqref{nonsymU} above). { We conclude by pointing out that it is an open and challenging problem to determine the possible limits in distribution of general sequences of degenerate $U$-statistics of a fixed order $\geq 2$. Our Theorem \ref{maintheo} demonstrates the remarkable fact that the Gamma distribution emerges naturally for sequences of degenerate $U$-statistics of an arbitrary order, under minimal moment conditions and provided a Lindberg-type assumption is verified.}

\end{itemize}

}

\section{Stein's method and exchangeable pairs \\ for Gamma approximations}\label{stein}
\subsection{Main estimates for Gamma approximations} {\it Stein's method} is a popular technique for estimating the distance between the distribution of some given random variable $W$ and a usually better understood target distribution. It was first developed by Stein \cite{St72} for the 
standard normal distribution and has by now been extended to many other univariate distributions, like the Poisson (see e.g. \cite{Ch75}, \cite{AGG89} or \cite{BHJ}), the Exponential (see e.g. \cite{CFR11}, \cite{PekRol11} and \cite{FulRos13}), the Beta (\cite{GolRei13} and \cite{DoeBeta}), the Gamma (\cite{Luk}, \cite{Pick}, \cite{Gau13} and  \cite{GRP}) and the Variance-Gamma (see \cite{Gau14}) distributions.\\
Stein's method for the Gamma distribution was first considered by Luk \cite{Luk}. There it was found that a real random variable $X$ has the $\Gamma(r,\lambda)$ distribution if and only if 
\begin{equation*}
 \E\bigl[Xf'(X)\bigl]=-\E\bigl[(r-\lambda X)f(X)\bigr]
\end{equation*}
holds for a sufficiently rich class of functions $f$. Following Stein's seminal idea this led him to the Gamma Stein equation \eqref{steineq1}, which, given the test function $h$ on $\R$ with $\E\abs{h(X_{r,\lambda})}<\infty$, is to be solved for $f$.   
Usually, this equation is only considered and solved on the support $[0,\infty)$ of $\Gamma(r,\lambda)$ but for our purposes we will need a solution $f_h$ to \eqref{steineq1} which is defined on the whole real line. Here, 
by a solution of \eqref{steineq1} we mean a function $f$ on $\R$ which is locally absolutely continuous and which satisfies \eqref{steineq1} at those points at which it is in fact differentiable. 
Given such a function, contrary to the usual convention, we define $f'$ at the non-zero points of non-differentiability of $f$ by \eqref{steineq1}. If $f$ is not differentiable at $0$, then, for definiteness, we let
$f'(0):=0$. For a test function $h$ as above, a solution $f_h$ to \eqref{steineq1} and a given real-valued random variable $W$ we thus obtain 
\begin{equation}\label{steinid}
 \babs{\E[h(W)]-\E[h(X_{r,\lambda})]}=\babs{\E\bigl[Wf_h'(W)+(r-\lambda W)f_h(W)\bigr]}\,,
 \end{equation}
whenever the right hand side is well-defined. As it turns out, the right hand side of \eqref{steinid} may often be efficiently bounded by means of some additional tool exploiting the structure of the random quantity $W$. This might be a similar characterization for the law of $W$, an integration by parts formula on the space where $W$ is defined, or a suitable coupling construction. \\
In any case, in order to bound the right hand side of \eqref{steinid} it is crucial to have smoothness bounds on the solution $f_h$ of \eqref{steineq1} in terms of the test function $h$. One of the theoretical contributions of this paper is 
to provide a new set of such bounds which are valid for the solution $f_h$ on the whole real line, not just on $[0,\infty)$. This is essential for our purposes, as the random variables $W$ we consider in our applications do not necessarily have range included in the positive axis. Another consequence of our new bounds is an improvement of Theorem 2.1 from \cite{PTh} and its consequences which deals with the Gamma approximation of functionals of a Poisson random measure.

To deal with our main application in this paper, we develop the technique of \textit{exchangeable pairs} in the context of Gamma approximation. This coupling construction lies at the heart of Stein's method and was first considered for normal approximation in Stein's celebrated monograph \cite{St86}. In the recent paper \cite{DP16} the authors applied it to the uni- and multivariate approximation of (vectors of) degenerate $U$-statistics. In particular, we 
were able to derive a complete quantitative extension of a famous CLT by de Jong \cite{deJo90}. 

In what follows, for a function $f$ on $\R$, we denote by 
\begin{equation*}
 \fnorm{f'}:=\sup_{x\not=y}\frac{\abs{f(x)-f(y)}}{\abs{x-y}}\in[0,\infty)\cup\{+\infty\}
\end{equation*}
its minimum Lipschitz constant. This notation does not cause any confusion as it coincides with the supremum norm of the derivative of $f$ whenever $f$ is differentiable.
Similarly, if $f$ is $n$-times differentiable for some $n\geq1$, we denote by $\fnorm{f^{(n+1)}}$ the minimum Lipschitz constant of $f^{(n)}$. 
We can now state our new smoothness estimates for the solution $f_h$ of \eqref{steineq1} on $\R$. We defer the proof of the next theorem to the end of this section. 

\begin{theorem}\label{bounds1}
\begin{enumerate}[{\normalfont (a)}]
 \item Let $h$ be Lipschitz-continuous on $\R$. %with minimum Lipschitz constant $\fnorm{h'}$. 
 Then, there exists a Lipschitz-continuous solution $f_h$ of \eqref{steineq1} on $\R$ which satisfies the bounds
\begin{equation}\label{e:hargrove}
\fnorm{f_h}\leq\lambda^{-1}\fnorm{h'}\quad\text{and}\quad \fnorm{f_h'}\leq 2\max\Bigl(1,\frac{1}{r}\Bigr)\fnorm{h'}\,.
\end{equation}
 \item Suppose that $h$ is continuously differentiable on $\R$ and that both $h$ and $h'$ are Lipschitz-continuous. Then, the solution $f_h$ of \eqref{steineq1} from {\normalfont (a)} is continuously differentiable and its derivative $f_h'$
 is Lipschitz-continuous with minimum Lipschitz constant
\begin{equation}\label{e:earfood}
\fnorm{f_h''}\leq 4\lambda\max\Bigl(1,\frac{1}{r}\Bigr)\fnorm{h'}+2\fnorm{h''}\,.
\end{equation}
\end{enumerate}
\end{theorem}

\begin{remark}
\begin{enumerate}[(a)]
\item By inspection of the proof of Theorem \ref{bounds1}, one sees that the following refinement of \eqref{e:hargrove} holds: writing $f_h^+$ and $f_h^{-}$ for the restriction of $f_h$ to $\R_+$ and $\R_{-}$, respectively, one has that
\begin{equation}\label{e:finer1}
\fnorm{(f^+_h)'}\leq 2\fnorm{h'}, \quad \mbox{and} \quad  \fnorm{(f^{-}_h)'}\leq \frac{2}{r}\fnorm{h'}
\end{equation}
We will see in Remark \ref{r:explosion} that, in principle, the quantity $2/r$ in the previous estimate {\it cannot} be replaced by a factor that is uniformly bounded in $r$.

\item Using the iterative technique for bounding higher derivatives of solutions to Stein equations from \cite{DoeBeta} which is further detailed in the recent paper \cite{DGV}, from the bound given in Theorem \ref{bounds1} (b), we can easily derive the bound 
\begin{align}\label{horder}
 \fnorm{f_h^{(k)}}&\leq2^k\lambda^{k-1}(k-1)!\max\Bigl(1,\frac{1}{r}\Bigr)\fnorm{h'}\notag\\
 &\;+\sum_{j=0}^{k-2} 2^{j+1}\lambda^j\frac{(k-1)!}{(k-j-1)!}\fnorm{h^{(k-j)}}\,,
\end{align}
valid for each $k\geq1$ and each $(k-1)$-times differentiable test function $h$, whose first $k-1$ derivatives are bounded and whose $(k-1)$st derivative is Lipschitz-continuous. 
\item In the recent paper \cite{GRP}, the authors proved that for each $k\geq1$ and each test function $h$ from some specific sub-class $\mathcal{C}_{\lambda, k}$ of all $(k-1)$-times differentiable functions such that $h^{(k-1)}$ is still absolutely continuous, the following bound holds: 
\begin{equation}\label{gprbound}
 \sup_{x>0}\abs{f_h^{(k)}(x)}\leq\frac{2}{r+k}\Bigl(3\sup_{x>0}\abs{h^{(k)}(x)}+2\lambda \sup_{x>0}\abs{h^{(k-1)}(x)}\Bigr)
\end{equation}
Note that, as opposed to the bounds from Theorem \ref{bounds1} or the bound \eqref{horder}, this bound converges to $0$ whenever the shape parameter $r$ of the Gamma distribution goes to $\infty$, which can be beneficial for certain applications as demonstrated in 
\cite{GRP}. However, the bounds given in the present paper are valid on the whole real line and are thus applicable to a broader class of applications. 
{We conjecture that there do exist positive, finite constants $C_{r,\lambda}^{(1)}$ and 
$C_{r,\lambda}^{(2)}$ with 
\begin{equation*}
 \fnorm{f_h'}\leq C_{r,\lambda}^{(1)}\fnorm{h'}\quad\text{and}\quad\fnorm{f_h''}\leq C_{r,\lambda}^{(2)}\bigl(\fnorm{h'}+\fnorm{h''}\bigr) 
 \end{equation*}
 such that $\lim_{r\to\infty}C_{r,\lambda}^{(1)}=\lim_{r\to\infty}C_{r,\lambda}^{(2)}=0$. These may be derived by a more careful investigation of the solutions $f_h$ on the support interval $[0,\infty)$.
 On the other hand, as already mentioned (see again Remark \ref{r:explosion}), the property $\lim_{r\downarrow0}C_{r,\lambda}^{(1)}=\infty$ is inevitable, as opposed to the bounds \eqref{gprbound} for the solutions on $(0,\infty)$.}
 \end{enumerate}
\end{remark}

\subsection{Targeting the centered Gamma distribution} Next, we transfer the bounds found in Theorem \ref{bounds1} to the centered Gamma distribution $\Gammabar(\nu)$ of $Z_\nu$ and state an off-the-shelf result, which bounds the distance between the distribution of a given random variable 
$W$ and $\Gammabar(\nu)$ in terms of an exchangeable pair. To the best of our knowledge this approach has not been considered in the context of Gamma approximation so far. 
The Stein equation for $\Gammabar(\nu)$ we use is given by 
\begin{equation}\label{steineq2}
 2(x+\nu)f'(x)-xf(x)=h(x)-\E\bigl[h(Z_\nu)\bigr]\,,
\end{equation}
where $h$ is Borel-measurable on $\R$ with $\E\abs{h(Z_\nu)}<\infty$. 

\begin{theorem}\label{bounds2}
 \begin{enumerate}[{\normalfont (a)}]
 \item Let $h$ be Lipschitz-continuous on $\R$. %with minimum Lipschitz constant $\fnorm{h'}$. 
 Then, there exists a Lipschitz-continuous solution $f_h$ of \eqref{steineq2} on $\R$ which satisfies the bounds
 { \[\fnorm{f_h}\leq\fnorm{h'}\quad\text{and}\quad \fnorm{f_h'}\leq\max\Bigl(1,\frac{2}{\nu}\Bigr) \fnorm{h'}\,.\]}
 \item Suppose that $h$ is continuously differentiable on $\R$ such that both $h$ and $h'$ are Lipschitz-continuous. Then, there is a continuously differentiable solution $f_h$ of \eqref{steineq2} on $\R$ whose 
 derivative $f_h'$ is Lipschitz-continuous with minimum Lipschitz constant
{ \[\fnorm{f_h''}\leq \max\Bigl(1,\frac{2}{\nu}\Bigr)\fnorm{h'}+\fnorm{h''}\,.\]}
\end{enumerate}
\end{theorem}

\begin{remark}
\begin{enumerate}[(a)]
\item Plainly, refined bounds analogous to \eqref{e:finer1} can be obtained for the function $f_h$ appearing in the statement of Theorem \ref{bounds2}.
\item In \cite{PTh} the slightly different Stein equation \begin{equation}\label{steineq3}
 2(x+\nu)_+f'(x)-xf(x)=h(x)-\E\bigl[h(Z_\nu)\bigr]
\end{equation}
with $g_+:=\max(g,0)$ was used. It turns out that the solution $f_h$ of \eqref{steineq2} from Theorem \ref{bounds2} has better smoothness properties at the singularity point $x=-\nu$ of the Stein equation than the solution 
of \eqref{steineq3} considered in \cite{PTh}. This makes it possible for us to improve the bounds on Gamma approximation on the Poisson space provided there. Furthermore, for the application to $U$-statistics in the present paper, it is essential to have a linear coefficient function for $f'$ in the Stein equation. This will become clear from the proof of Theorem \ref{maintheo} in Section \ref{proofmt}.
\end{enumerate}
\end{remark}

\begin{proof}[Proof of Theorem \ref{bounds2}]
 Given $h$ with $\E\abs{h(Z_\nu)}<\infty$ we define $h_1$ by $h_1(x):=h(2x-\nu)$. It is easy to see that if $g_h$ is the solution to \eqref{steineq1} with $h$ replaced by $h_1$ from Theorem \ref{bounds1}, 
 then the function $f_h$ with 
 \[f_h(x):=\frac{1}{2}g_h\Bigl(\frac{x+\nu}{2}\Bigr)\]
 solves \eqref{steineq2}. Furthermore, from Theorem \ref{bounds1} we obtain the bounds
 { \begin{align*}
  \fnorm{f_h}&=\frac{1}{2}\fnorm{g_h}\leq\frac12\fnorm{h_1'}=\fnorm{h'}\,,\\
  \fnorm{f_h'}&=\frac14\fnorm{g_h'}\leq\frac14\cdot 2\max\Bigl(1,\frac{2}{\nu}\Bigr)\fnorm{h_1'}=\max\Bigl(1,\frac{2}{\nu}\Bigr)\fnorm{h'}\quad\text{and}\\
  \fnorm{f_h''}&=\frac18\fnorm{g_h''}\leq\frac18\Bigl(2\fnorm{h_1''}+4\max\Bigl(1,\frac{2}{\nu}\Bigr)\fnorm{h_1'}\Bigr)\\
  &=\fnorm{h''}+\max\Bigl(1,\frac{2}{\nu}\Bigr)\fnorm{h'}\,.
 \end{align*}}
\end{proof}

\subsection{Exchangeable pairs} Let $W,W'$ be identically distributed real-valued random variables defined on the same probability space $(\Omega,\F,\Prob)$ such that $\E[W^2]<\infty$. Assume that $\G$ is a sub-$\sigma$-field of $\F$ such 
that $\sigma(W)\subseteq\G$. Given a real number $\lambda>0$ we define the random variables $R$ and $S$ via the \textit{regression equations}
\begin{align}
 \frac{1}{\lambda}\E\bigl[W'-W\,\bigl|\,\G\bigr]&=-W+R\quad\text{and}\label{linreg}\\
 \frac{1}{2\lambda}\E\bigl[(W'-W)^2\,\bigl|\,\G\bigr]&=2(W+\nu)+S\label{regprop2}\,.
\end{align}

In many cases of interest Equation \eqref{linreg} holds with $R=0$ for some (unique) $\lambda>0$ but as was exemplified in \cite{RiRo97} it is convenient to allow for a non-trivial remainder term $R$, in general. 
From Proposition 3.19 and Remark 3.10 of \cite{DoeBeta}, as well as from the bounds given by Theorem \ref{bounds2} we obtain the following new plug-in result for centered Gamma approximation { which can be seen as a generalization of Theorem 2.1 of \cite{FulRos13} dealing with exponential approximation}. This theorem will play a major role in the proof of Theorem \ref{maintheo}.  

\begin{theorem}\label{plugin}
Let $W$ and $W'$ be as above and assume that $h$ is continuously differentiable on $\R$ such that both $h$ and $h'$ are Lipschitz-continuous. Then, we have the bound
\begin{align*}
 \Babs{\E\bigl[h(W)\bigr]-\E\bigl[h(Z_\nu)\bigr]}&\leq\fnorm{h'}\bigl(\max(1,2\nu^{-1})\E\abs{S}+\E\abs{R}\bigr)\\
 &\;+\frac{\max\bigl(1,\frac{2}{\nu}\bigr)\fnorm{h'}+\fnorm{h''}}{6\lambda}\E\babs{W'-W}^3\,.
\end{align*}
If, moreover, $\E[W^2]=2\nu$ and \eqref{linreg} holds with $R=0$, then, since $\E[S]=0$ in this case, we also have the bound
\begin{align}\label{pluginbound}
 \Babs{\E\bigl[h(W)\bigr]-\E\bigl[h(Z_\nu)\bigr]}&\leq\max\Bigl(1,\frac{2}{\nu}\Bigr)\fnorm{h'}\sqrt{\Var(S)}\notag\\
 &\;+\frac{\max\bigl(1,\frac{2}{\nu}\bigr)\fnorm{h'}+\fnorm{h''}}{6\lambda}\E\babs{W'-W}^3\,.
\end{align}

\end{theorem}

\subsection{Proofs} The following two lemmas will be needed for the proof of Theorem \ref{bounds1}.
\begin{lemma}\label{unique}
 Let $h$ be a  Borel-measurable function $h$ on $\R$ with $\E\abs{h(X_{r,\lambda})}<\infty$. Then, on each of the two intervals $(-\infty,0)$ and $(0,\infty)$, 
 there exists at most one bounded solution $f$ of \eqref{steineq1}.
\end{lemma}

\begin{proof}
Let $f$ be a bounded solution of \eqref{steineq1} on $(-\infty,0)$. 
The solutions of the corresponding homogeneous equation are given by the constant multiples of the function 
\[{\psi(x):=\abs{x}^{-r}e^{\lambda x}\,,\quad x<0\,.}\]
Thus, if $g$ is another solution of \eqref{steineq1} on $(-\infty,0)$, then there is a constant $c\in\R$ such that 
\[g=f+c\psi\,.\]
As $\psi(0-)=-\infty$ and $\sup_{x<0}\abs{f(x)}<\infty$, it follows that $g$ can only be bounded if $c=0$, i.e. if $g=f$. The proof for the interval $(0,\infty)$ is very similar.
\end{proof}

\begin{lemma}\label{fprimelemma}
 Let $a<b$ be real numbers and let $f:[a,b]\rightarrow\R$ be a function having the following properties:  
\begin{enumerate}[{\normalfont (a)}]
 \item $f$ is continuous on $[a,b]$.
 \item $f_{|[c,b]}$ is absolutely continuous for each $a<c<b$ (and, hence, $f$ is $\lambda$-almost everywhere differentiable on $(a,b]$).
 \item There is some $a<d<b$, a set $A\subseteq (a,d)$ at each of whose points $f$ is differentiable with $\lambda\left((a,d)\setminus A\right)=0$ and a real number $\gamma$ such 
 that $\lim_{n\to\infty} f'(x_n)=\gamma$ for each sequence $(x_n)_{n\in\N}$ lying in $A$ with $\lim_{n\to\infty}x_n=a$.
\end{enumerate}
Then, $f$ is absolutely continuous on $[a,b]$ and differentiable at $a$ with $f'(a)=\gamma$. Furthermore, the function $f'$ restricted to $A\cup\{a\}$ is continuous at $a$.
\end{lemma}

\begin{proof}
 A proof can be found in the appendix of the thesis \cite{Doe12c}.\\
\end{proof}

We are now in the position to prove Theorem \ref{bounds1}.

\begin{proof}[Proof of Theorem \ref{bounds1}]
 It suffices to prove the theorem for the case $\lambda=1$. In fact, it is easy to see that if $g$ solves 
 \begin{equation*}
 xg'(x)+(r-x)g(x)=h_1(x)-\E[h_1(X_{r,1})]\,,
\end{equation*}
where 
\begin{equation*}
 h_1(x):=h(x/\lambda)\,,
\end{equation*}
then $f(x):=g(\lambda x)$ solves \eqref{steineq1}. Taking into account the identities 
\begin{equation*}
 f^{(k)}(x)=\lambda^k g^{(k)}(\lambda x)\,,\quad\fnorm{h_1'}=\lambda^{-1}\fnorm{h'}\quad\text{and}\quad \fnorm{h_1''}=\lambda^{-2}\fnorm{h''}
\end{equation*}
then yields the bounds for general $\lambda>0$. So let us assume for the rest of the proof that $\lambda=1$.
For notational convenience we will also write $p_r$ for $p_{r,1}$, $F_r$ for $F_{r,1}$ and $X_r$ for $X_{r,1}$.

We first prove (a). 
As $h$ is continuous, it is known (see e.g. \cite{DoeBeta}, Proposition 3.8) that the function $f_h:(0,\infty)\rightarrow\R$ with 
\begin{align}\label{sol1}
 f_h^+(x):&=\frac{1}{xp_{r}(x)}\int_0^x\bigl(h(t)-\E[h(X_{r})]\bigr)p_{r}(t)dt\notag\\
 &=-\frac{1}{xp_{r}(x)}\int_x^\infty\bigl(h(t)-\E[h(X_{r})]\bigr)p_{r}(t)dt
\end{align}
is a continuously differentiable solution to \eqref{steineq1} on $(0,\infty)$ which can be continuously extended to $[0,\infty)$ by letting 
\begin{equation}\label{limzerop}
 f_h^+(0):=\frac{h(0)-\E[h(X_{r})]}{r}\,.
\end{equation}
For a Lipschitz-continuous test function $h$ we know from Corollary 3.15 of \cite{DoeBeta} that $\sup_{x\geq0}\abs{f_h^+(x)}\leq \fnorm{h'}$ and that for each $x>0$ we have
\begin{equation*}
 \abs{(f_h^+)'(x)}\leq2\fnorm{h'}\frac{\int_0^x F_{r}(t)dt\int_x^\infty(1-F_{r}(t))dt}{x^2p_{r}(x)}=:2\fnorm{h'}S_{r}(x)\,.
\end{equation*}

We bound $S_r(x)$ for $0<x\leq r$ and for $x>r$ separately. Assume $x>r$. From Fubini's theorem we conclude that
\begin{align*}
 \int_x^\infty(1-F_r(t))dt&=\int_x^\infty \int_t^\infty p_r(s)ds dt=\int_x^\infty (s-x)p_r(s)ds\\
&=\int_x^\infty (s-r)p_r(s)ds+(r-x)(1-F_r(x)) \\
&=xp_r(x)+(r-x)(1-F_r(x))\leq xp_r(x)\,.
 %\int_x^\infty (s-r)p_r(s)ds=xp_r(x)\,.
\end{align*}
Also note that
\begin{equation}\label{eq1}
 \int_0^x F_r(t)dt\leq xF_r(x)\,,\quad x\geq0\,,
\end{equation}
as $F_r$ is nondecreasing. Hence, we obtain that 
\begin{equation*}
 S_r(x)\leq\frac{xF_r(x)\cdot xp_r(x)}{x^2p_r(x)}=F_r(x)\leq 1\,,\quad r< x\,.
\end{equation*}
Now let $0<x\leq r$. Note first that 
\begin{equation}\label{eq2}
 \int_x^\infty(1-F_r(t))dt\leq \int_0^\infty(1-F_r(t))dt=\E[X_{r}]=r\,.
\end{equation}
Next, consider 
\begin{align*}
 R(x):&=\frac{\int_0^x F_r(t)dt}{x^2p_r(x)}=\frac{xp_r(x)-(r-x)F_r(x)}{x^2p_r(x)}\quad\text{with}\\
 R'(x)&=\frac{xF_r(x)-(1+r-x)\int_0^x F_r(t)dt}{x^3 p_r(x)}=:\frac{N(x)}{x^3p_r(x)}\,.
\end{align*}
{
We claim that $N(x)\geq0$ for all $x\geq0$, which implies that $R$ is increasing. Note first that $N(0)=0$. Also, we have 
\begin{align*}
 N'(x)&=F_r(x)+xp_r(x)+\int_0^xF_r(t)dt-(1+r-x)F_r(x)\\
 &=xp_r(x)-(r-x)F_r(x)+\int_0^xF_r(t)dt\\
 &=2\int_0^xF_r(t)dt\geq0\,.
\end{align*}
Hence, $N$ is increasing and, thus, $N(x)\geq0$ for all $x\geq0$. This implies that 
\begin{equation*}
 \sup_{0<x\leq r}R(x)=R(r)=\frac{1}{r}
\end{equation*}
and from \eqref{eq2} we conclude that
\begin{equation*}
 \sup_{0<x\leq r}S_r(x)\leq r\sup_{0<x\leq r}R(x)=\frac{r}{r}=1\,.
 \end{equation*}
}
Thus, we have proved that 
\begin{equation*}
 \sup_{x>0}\abs{(f_h^+)'(x)}\leq 2\fnorm{h'}\,.
\end{equation*}

In order to solve \eqref{steineq1} on $(-\infty,0)$ we use the theory developed in Section 2.4 of \cite{Doe12c} (see also the unpublished manuscript \cite{Doe12a}).
There it is shown that a solution to \eqref{steineq1} on $(-\infty,0)$ is given by 
\begin{align}\label{sol2}
 f_h^-(x)&=\exp\bigl(-G_l(x)\bigr)\int_0^x \bigl(h(t)-\E[h(X_{r})]\bigr)\frac{\exp\bigl(G_l(t)\bigr)}{t}dt\notag\\
 &=\frac{1}{xq_l(x)}\int_0^x \bigl(h(t)-\E[h(X_{r})]\bigr)q_l(t)dt\,,
\end{align}
where $G_l$ is an arbitrary primitive function of $x\mapsto\frac{r-x}{x}$ on $(-\infty,0)$ and 
\begin{equation*}
 q_l(x):=\frac{\exp(G_l(x))}{x}\,,\quad x<0\,.
\end{equation*}
Also $f_h^-$ can be continuously extended to $(-\infty,0]$ by letting 
\begin{equation}\label{limzerom}
f_h^-(0):=\frac{h(0)-\E\bigl[h(X_r)\bigr]}{r}\,. 
\end{equation}
This follows from Proposition 2.4.28 of \cite{Doe12c} (or Proposition 2.22 of \cite{Doe12a}). Again, by the continuity of $h$, it is easy to see that $f_h^-$ is continuously differentiable on $(-\infty,0)$.
We choose 
\begin{equation*}
 G_l(x):=r\log(-x)-x\,,\quad x<0\,,
\end{equation*}
yielding 
\begin{equation*}
 q_l(x)=-(-x)^{r-1}e^{-x}\,,\quad x<0\,.
\end{equation*}
Note that $q_l(x)<0$ for all $x\in (-\infty,0)$. Furthermore, we define the function $Q_l$ on $(-\infty,0)$ by 
\begin{equation*}
 Q_l(x):=\int_0^x q_l(t)dt=\int_x^0(-q_l(t))dt>0\,.
\end{equation*}
By taking its derivative we see that $Q_l$ is decreasing on $(-\infty,0)$. From Corollary 2.4.36 in \cite{Doe12c} (or Corollary 2.28 of \cite{Doe12a}) we have 
\[\sup_{x<0}\abs{f_h^-(x)}\leq \fnorm{h'}\]
and, for each $x\in(-\infty,0)$, that 
{
\begin{align}
 \abs{(f_h^-)'(x)}&\leq 2\fnorm{h'}\frac{(r-x)\Bigl(-xQ_l(x)+\int_0^xtq_l(t)dt\Bigr)}{-x^2q_l(x)}\notag\\
 &=2\fnorm{h'}\frac{(r-x)\int_x^0 Q_l(t)dt}{-x^2q_l(x)}\notag\\
 %&= 2r\fnorm{h'}\frac{\int_x^0 Q_l(t)dt}{-x^2q_l(x)}+2\fnorm{h'}\frac{\int_x^0 Q_l(t)dt}{xq_l(x)}\notag\\
 &=:2\fnorm{h'}U(x)\label{eq3}\,.
\end{align}
Define the function $T$ as well as $q_u$ and $Q_u$ on $(0,\infty)$ by $T(y):=U(-y)$,\\
$q_u(y):=-q_l(-y)=y^{r-1}e^y$ and 
\begin{align*}
 Q_u(y)&:=Q_l(-y)=\int_0^{-y} q_l(t)dt=-\int_0^yq_l(-s)ds\\
 &=\int_0^yq_u(s)ds\,.
\end{align*}
Then, we have 
\begin{align*}
 \int_{-y}^0Q_l(t)dt&=-\int_y^0Q_l(-s)ds=\int_0^y Q_u(s)ds=\int_0^y\int_0^tq_u(t)dt\,ds\\
 &=\int_0^y(y-t)q_u(t)dt=yQ_u(y)-\int_0^y tq_u(t)dt
\end{align*}
as well as the representations
\begin{align}
 T(y)&=\frac{(r+y)\int_0^yQ_u(s)ds}{y^2 q_u(y)}=\frac{(r+y)\int_0^yQ_u(s)ds}{y^{r+1}e^y}\label{eq4a}\\
 &=\frac{(r+y)\int_0^y(y-t)q_u(t)dt}{y^{r+1}e^y}=\frac{(r+y)(yQ_u(y)-\int_0^ytq_u(t)dt)}{y^{r+1}e^y}\label{eq4b}
\end{align}
Note that, using de l'H\^{o}pital's rule, we obtain
\begin{align}\label{eq4c}
 \lim_{y\downarrow0}T(y)&=r\lim_{y\downarrow0}\frac{\int_0^yQ_u(s)ds}{y^{r+1}e^y} 
 =r\lim_{y\downarrow0}\frac{Q_u(y)}{y^re^y(r+1+y)}\notag\\
 &=\frac{r}{r+1}\lim_{y\downarrow0}\frac{Q_u(y)}{y^re^y}=\frac{r}{r+1}\lim_{y\downarrow0}\frac{y^{r-1}e^y}{y^{r-1}e^y(r+y)}\notag\\
 &=\frac{r}{r+1}\lim_{y\downarrow0}\frac{1}{r+y}=\frac{1}{r+1}
\end{align}
as well as
\begin{align}\label{eq4d}
 \lim_{y\to\infty}T(y)&=\lim_{y\to\infty}\frac{r+y}{y}\cdot\lim_{y\to\infty}\frac{\int_0^yQ_u(s)ds}{y^{r}e^y} 
 =\lim_{y\to\infty}\frac{y^{r-1}e^r}{(r+y)y^{r-1}e^y}\notag\\
 &=\lim_{y\to\infty}\frac{1}{r+y}=0\,.
\end{align}
By the continuity of $T$, from \eqref{eq4c} and \eqref{eq4d} we already conclude that 
\begin{equation}\label{uandt}
\sup_{x<0}U(x)=\sup_{y>0}T(y)<\infty\,.
\end{equation}
Hence, it remains to deal with the local maxima of the function $T$. Note that 
\begin{align*}
 T'(y)&=\frac{y(r+y)Q_u(y)+\int_0^yQ_u(s)ds\;\bigl(y-(r+1+y)(r+y)\bigr)}{y^{r+2}e^y}\\
 &=\frac{y(r+y)Q_u(y)-\int_0^yQ_u(s)ds\;(y^2+2ry+r^2+r)}{y^{r+2}e^y}\,.
\end{align*}
If $y_0\in(0,+\infty)$ is a locally maximal point of $T$, then $T'(y_0)=0$. This implies that   
\begin{align*}
 \int_0^{y_0}Q_u(s)ds&=\frac{y_0(r+y_0)Q_u(y_0)}{y_0^2+2ry_0+r^2+r}\quad\text{and}\\
 T(y_0)&=\frac{(r+y_0)^2Q_u(y_0)}{y_0^re^{y_0}(y_0^2+2ry_0+r^2+r)}\leq\frac{Q_u(y_0)}{y_0^re^{y_0}}\,.
\end{align*}
Define the function $T_2$ on $(0,\infty)$ by 
\begin{equation*}
 T_2(y):=\frac{Q_u(y)}{y^re^{y}}\,.
\end{equation*}
Then, the above discussion as well as \eqref{eq4c} and \eqref{eq4d} show that 
\begin{align}\label{tbound1}
 \sup_{y>0}T(y)&\leq\max\Biggl(\lim_{y\downarrow0}T(y),\,\lim_{y\to\infty}T(y),\,\sup_{y>0}T_2(y)\Biggr)\notag\\
 &=\max\Biggl(\frac{1}{r+1},\,\sup_{y>0}T_2(y)\Biggr)\,.
\end{align}
Using the explicit expression of $Q_u$, as well as the elementary estimate $1\leq e^s\leq e^y$ for every $0\leq s\leq y$, one deduces immediately that
\begin{equation}\label{tbound}
\frac{ e^{-y}}{r}\leq  T_2(y)\leq \frac1r, \quad y>0.
\end{equation}

%We have, using de l'H\^{o}pital's rule, that 
%\begin{align}\label{lim1t2}
% \lim_{y\downarrow0}T_2(y)&=\lim_{y\downarrow0}\frac{y^{r-1}e^y}{y^{r-1}e^y(r+y)}=\lim_{y\downarrow0}\frac{1}{r+y}=\frac{1}{r}
%\end{align}
%as well as 
%\begin{align}\label{lim2t2}
% \lim_{y\to\infty}T_2(y)&=\lim_{y\to\infty}\frac{y^{r-1}e^y}{y^{r-1}e^y(r+y)}=\lim_{y\to\infty}\frac{1}{r+y}=0\,.
%\end{align}
%We claim that $T_2$ is decreasing on $(0,\infty)$. We have 
%\begin{equation*}
% T_2'(y)=\frac{y^re^y-(r+y)Q_u(y)}{y^{r+1}e^y}=:\frac{D(y)}{y^{r+1}e^y}\,.
%\end{equation*}
%Note that $\lim_{y\downarrow0}D(y)=0$. Furthermore, 
%\begin{equation*}
% D'(y)=(y+r)y^{r-1}e^y-Q_u(y)-(r+y)y^{r-1}e^y=-Q_u(y)<0
%\end{equation*}
%for all $y>0$. Hence, $D$ is decreasing on $(0,\infty)$ implying that $D(y)<0$ for all $y>0$. This implies that also $T_2'(y)<0$ for all $y<0$ proving that $T_2$ is decreasing. Hence, 
%\begin{equation*}
% \sup_{y>0}T_2(y)= \lim_{y\downarrow0}T_2(y)=\frac{1}{r}
%\end{equation*}
%and hence, by \eqref{tbound1} also
%\begin{equation}\label{tbound}
% \sup_{y>0}T(y)\leq\frac{1}{r}\,.
%\end{equation}
%
%
%
%
}

\noindent Hence, from \eqref{eq3}, \eqref{uandt} and \eqref{tbound} we conclude that for all $x\in(-\infty,0)$ we have
\begin{equation}\label{eq5}
 \abs{(f_h^-)'(x)}\leq\frac{2}{r}\fnorm{h'}\,.
\end{equation}
Now, we define the function $f_h$ on $\R$ by letting 
\begin{equation*}
 f_h(x):=\begin{cases}
          f_h^-(x)\,,& x\leq0\\
          f_h^+(x)\,,&x\geq0\,.
         \end{cases}
\end{equation*}
Note that \eqref{limzerop} and \eqref{limzerom} imply that $f_h$ is well-defined. As the continuous concatenation of two Lipschitz-continuous functions, we recognize $f_h$ to be Lipschitz-continuous on $\R$ with 
\begin{align*}
 \fnorm{f_h}&\leq\max\Bigl(\sup_{x\leq0}\abs{f_h^-(x)}, \sup_{x\geq0}\abs{f_h^+(x)}\Bigr)\leq \fnorm{h'}\quad\text{and}\\
 \fnorm{f_h'}&\leq \max\Bigl(\sup_{x<0}\abs{(f_h^-)'(x)}, \sup_{x>0}\abs{(f_h^+)'(x)}\Bigr)\leq \max\Bigl(2,\frac{2}{r}\Bigr)\fnorm{h'}\,.
\end{align*}
This finishes the proof of (a).

To prove (b) we assume that $h$ is continuously differentiable and that both $h$ and its derivative $h'$ are Lipschitz-continuous.
The identity 
\begin{equation*}
 xf_h'(x)+(r-x)f_h(x)=h(x)-\E\bigl[h(X_{r})\bigr]
\end{equation*}
implies that $f_h$ is continuously differentiable on both of the two intervals $(-\infty,0)$ and $(0,\infty)$ and differentiating yields that $f_h'$ solves  
\begin{equation}\label{steineqprime}
 xg'(x)+(r+1-x)g(x)=h'(x)+f_h(x)=:h_2(x)
\end{equation}
on both intervals $(-\infty,0)$ and $(0,\infty)$. Note that \eqref{steineqprime} is the Stein equation for the distribution $\Gamma(r+1,1)$ corresponding to the test function $h_2$.
We already know from part (a) that $f_h'$ is bounded by $2\max(1,r^{-1})\fnorm{h'}$. Also, Proposition 3.17 of \cite{DoeBeta} implies that $h_2$ is centered with respect to the $\Gamma(r+1,1)$ distribution. Hence, as $h_2$ 
is Lipschitz-continuous with Lipschitz-constant 
\begin{equation}\label{lipconsh2}
 \fnorm{h_2'}\leq\fnorm{h''}+\fnorm{f_h'}\leq \fnorm{h''}+2\max(1,r^{-1})\fnorm{h'}
\end{equation}
we know from part (a) applied to the distribution $\Gamma(r+1,1)$ that there is a bounded solution $g_{h_2}$ of \eqref{steineqprime}. 
Since $f_h'$ is bounded on both of the intervals $(-\infty,0)$ and $(0,\infty)$, it thus follows from Lemma \ref{unique} that 
$f_h'(x)=g_{h_2}(x)$ for all $x\not=0$. Since $g_{h_2}$ is continuous at $0$ we know from the analogs of \eqref{limzerop} and \eqref{limzerom} for $g_{h_2}$ that 
\begin{align*}
\lim_{x\uparrow0}f_h'(x)&=\lim_{x\uparrow0}g_{h_2}(x)=g_{h_2}(0)=\lim_{x\to0}g_{h_2}(x)=\frac{h_2(0)-\E\bigl[h_2(X_{r+1,1})\bigr]}{r+1}=\frac{h_2(0)}{r+1}\\
 &=\frac{h'(0)}{r+1}+\frac{h(0)-\E[h(X_{r,1})]}{r(r+1)}\,,
\end{align*}
and, similarly we conclude that 
\begin{equation*}
 \lim_{x\downarrow0}f_h'(x)=\lim_{x\downarrow0}g_{h_2}(x)=g_{h_2}(0)=\frac{h'(0)}{r+1}+\frac{h(0)-\E[h(X_{r,1})]}{r(r+1)}\,.
\end{equation*}
By Lemma \ref{fprimelemma} this implies that $f_h$ is continuously differentiable on $\R$ with $f_h'(0)=g_{h_2}(0)$.
%where the last identity is by Convention \eqref{convfprime}. 
Hence, we have $f_h'=g_{h_2}$ and from part (a) and \eqref{lipconsh2} we conclude that $f_h'$ is Lipschitz-continuous with 
\begin{align*}
 \fnorm{f_h''}&=\fnorm{g_{h_2}'}\leq 2\max\Bigl(1,\frac{1}{r+1}\Bigr) \fnorm{h_2'} =2\fnorm{h_2'}\\
 &\leq 2\Bigl(\fnorm{h''}+2\max\bigl(1,r^{-1}\bigr)\fnorm{h'}\Bigr)\\
 &=2\fnorm{h''}+4\max\bigl(1,r^{-1}\bigr)\fnorm{h'}\,.
\end{align*}
\end{proof}

\begin{remark}\label{r:explosion}

As anticipated, the factor $2/r$ in \eqref{e:finer1} cannot be replaced by a quantity that is uniformly bounded in $r$. Indeed, from the proof of Proposition 2.4.35 in \cite{Doe12c} we have the representation 
\begin{align*}
 f_h'(x)=\frac{1}{-x^2q_l(x)}\Biggl((x-r)\int_x^0Q_l(s)h'(s)ds-\int_x^\infty (1-F(s))h'(s)ds\int_x^0Q_l(t)dt\Biggr)\,,
\end{align*}
whenever $h$ is Lipschitz-continuous and $x<0$. If we take $h(x)=\min(x,0)$, then we obtain
\begin{align*}
 \int_x^0Q_l(s)h'(s)ds&=\int_0^xQ_l(s)ds\quad\text{and}\\
 -\int_x^\infty (1-F(s))h'(s)ds\int_x^0Q_l(t)dt&=x\int_x^0Q_l(t)dt
\end{align*}
because $F(s)=0$ for all $s\leq0$. This gives
\begin{equation*}
 f_h'(x)=\frac{(2x-r)\int_x^0 Q_l(t)dt}{-x^2q_l(x)}
\end{equation*}
and, hence, straightforward estimates yield that, for $x<0$,
\begin{equation*}
 \abs{f_h'(x)}=\frac{(r-2x)\int_x^0 Q_l(t)dt}{-x^2q_l(x)}\geq e^x\frac{r-2x}{r(r+1)}\,.
\end{equation*}
In particular, this implies that
\begin{equation*}
 \sup_{x<0}\abs{f_h'(x)}\geq \abs{f_h'(-1/2)}\geq e^{-1/2}\frac{1}{r}\,.
\end{equation*}
By mollifying the Lipschitz function $h(x)=\min(x,0)$, one can also construct a function $h^*\in\mathcal{H}_2$ such that $\sup_{x<0}\abs{f_{h^*}'(x)}\geq e^{-1/2}\frac{1}{2r}$. This justifies the remark following Theorem \ref{maintheo}.
\end{remark}

\section{Proof of Theorem \ref{maintheo}}\label{proofmt}
We are going to apply the bound \eqref{pluginbound} from Theorem \ref{plugin} with $\fnorm{h'},\fnorm{h''}\leq1$ to the $\sigma$-field $\mathcal{G}=\sigma(X_1,\dots,X_n)$
and to the following exchangeable pair $(W,W')$ which has already been used in \cite{DP16}:  
 Let $Y:=(Y_j)_{1\leq j\leq n}$ be an independent copy of $X:=(X_j)_{1\leq j\leq n}$ and let $\alpha$ be uniformly distributed 
on $\{1,\dotsc,n\}$ such that $X,Y$ and $\alpha$ are jointly independent. Letting, for $j=1,\dotsc,n$, 
\begin{equation*}
 X_j':=\begin{cases}
        Y_j\,,&\text{if } \alpha=j\\
        X_j\,,&\text{if }\alpha\not=j
       \end{cases}
\end{equation*}
and
\begin{equation*}
 X':=(X_1',\dotsc,X_n')
\end{equation*}
it is easy to see that the pair $(X,X')$ is exchangeable.
Finally, as exchangeability is preserved under functions, letting 
\begin{align*}
 W':=\psi(X_1',\dotsc,X_n')&=\sum_{j=1}^n 1_{\{\alpha=j\}}\Biggl(\sum_{J:j\notin J}W_J+\sum_{J:j\in J}W_J^{(j)}\Biggr)\\
 &=:\sum_{J:\alpha\notin J}W_J+\sum_{J:\alpha\in J}W_J^{(\alpha)}\,,
\end{align*}
the pair $(W,W')$ is exchangeable. Here, for $J=\{j_1,\dotsc,j_m\}\subseteq[n]$ with $1\leq j_1<j_2<\dotsc<j_m\leq n$ and $j=j_k\in J$, we write
\begin{equation*}
 W_J^{(j)}:=\psi_J(X_{j_1},\dotsc,X_{j_{k-1}},Y_{j_k},X_{j_{k+1}},\dotsc,X_{j_m})\,.
\end{equation*}

From Lemma 2.2 of \cite{DP16} we know that \eqref{linreg} holds for the pair $(W,W')$ with $R=0$ and $\lambda=d/n$.
Also, if we denote by 
\begin{equation*}
 W^2=\sum_{\substack{M\subseteq[n]:\\\abs{M}\leq 2d}}U_M
\end{equation*}
the Hoeffding decomposition of $W^2$, then Lemma 2.6 of \cite{DP16} states that 
\begin{equation}\label{hdcond}
 \frac{n}{2d}\E\bigl[(W'-W)^2\,\bigl|\,X\bigr]=\sum_{\substack{M\subseteq[n]:\\\abs{M}\leq 2d-1}}a_M U_M\,,
\end{equation}
where 
\[a_M:=1-\frac{\abs{M}}{2d}\,,\quad M\subseteq[n]\quad\text{such that }\abs{M}\leq 2d\,.\]
Hence, we have the following Hoeffding decomposition \eqref{HDS} of $S$:
\begin{align}\label{HDS}
 S&=\frac{n}{2d}\E\bigl[(W'-W)^2\,\bigl|\,X\bigr]-2(W+\nu)=\sum_{\substack{M\subseteq[n]:\\\abs{M}\leq 2d-1}}\Bigl(1-\frac{\abs{M}}{2d}\Bigr)U_M -2W-2\nu\notag\\
&=\sum_{\substack{M\subseteq[n]:\\1\leq\abs{M}\leq 2d-1}}\Bigl(1-\frac{\abs{M}}{2d}\Bigr)U_M-2\sum_{J\in\D_d}W_J\notag\\
&=\sum_{\substack{M\subseteq[n]:\\1\leq\abs{M}\leq 2d-1\,,\\\abs{M}\not=d}}\Bigl(1-\frac{\abs{M}}{2d}\Bigr)U_M+\frac{1}{2}\sum_{J\in\D_d}\bigl(U_J-4W_J\bigr)\\
&=:S_1+\frac{1}{2}S_2\notag\,.
\end{align}
Here we have used that $U_\emptyset=\E[W^2]=2\nu$. By the orthogonality of the terms appearing in the Hoeffding decomposition we thus obtain that 
\begin{equation*}
 \Var(S)=\Var(S_1)+\frac{1}{4}\Var(S_2)\,.
\end{equation*}
From the orthogonality of the Hoeffding decomposition, we conclude that 
\begin{align}\label{3mform}
 \E[W^3]&=\sum_{J\in\D_d}\sum_{\substack{M\subseteq[n]:\\\abs{M}\leq 2d}}\E\bigl[W_J U_M\bigr]=\sum_{J\in\D_d}\E\bigl[W_J U_J\bigr]=\sum_{J\in\D_d}\Cov(W_J,U_J)\,.
\end{align}

Before we proceed, we need an auxiliary lemma which expresses the fourth moment of $W$ in terms of the exchangeable pair $(W,W')$. We first state a more general 
lemma, whose statement is in fact only a slight generalization of one of the key relations in Stein's method of exchangeable pairs (see \cite{St86} or \cite{CGS}) will be very useful. 
\begin{lemma}\label{exlemma1}
 Let $(W,W')$ be an exchangeable pair of real-valued random variables such that, for some $\lambda>0$, \eqref{linreg} is satisfied with $R=0$ and let $g$ be an absolutely continuous function on $\R$ with 
 $\E\bigl[\bigl(1+\abs{W}+\abs{W'}\bigr)\abs{g(W)}\bigr]<\infty$.% and $\E\abs{(W'-W)g(W)}<\infty$. 
 Then, 
 \begin{equation*}
 \E\bigl[Wg(W)\bigr]=\E\Bigl[g'(W) \frac{1}{2\lambda}\E\bigl[(W'-W)^2\,\bigl|\,\G\bigr]\Bigr] +E\,,%\frac{1}{2\lambda}\E\Bigl[(W'-W)^2\int_0^1\Bigl(g'\bigl(W+t(W'-W)\bigr)-g'(W)\Bigr)dt\Bigr]\,.
 %\frac{1}{2\lambda}\E\Bigl[(W'-W)^2\int_0^1 g'\bigl(W+t(W'-W)\bigr)dt\Bigr]\,.
 \end{equation*}
 where
 \begin{equation*}
  E:=\frac{1}{2\lambda}\E\Bigl[(W'-W)^2\int_0^1\Bigl(g'\bigl(W+t(W'-W)\bigr)-g'(W)\Bigr)dt\Bigr]
 \end{equation*}
is a remainder term.
\end{lemma}

\begin{lemma}\label{exlemma2}
 Let $(W,W')$ be an exchangeable pair of real-valued random variables in $L^4(\Prob)$ such that, for some $\lambda>0$, \eqref{linreg} is satisfied with $R=0$. Then, 
 \begin{align}
  \E[W^4]&=3\E\Bigl[W^2\frac{1}{2\lambda}\E\bigl[(W'-W)^2\,\bigl|\,\G\bigr]\Bigr]-\frac{1}{4\lambda}\E\bigl[(W'-W)^4\bigr]\label{4mformula}\quad\text{and}\\
  %&\;+\frac{1}{2\lambda}\E\bigl[(W'-W)^4\bigr]\label{4mformula}\\
  \E[W^3]&=2\E\Bigl[W\frac{1}{2\lambda}\E\bigl[(W'-W)^2\,\bigl|\,\G\bigr]\Bigr]\label{3mformula}\,.
 \end{align}
\end{lemma}

\begin{proof}
 The proof of \eqref{4mformula} applies Lemma \ref{exlemma1} with $g(x)=x^3$ leading to the remainder term
 \begin{align*}
  E&=\frac{3}{2\lambda}\E\Bigl[(W'-W)^2\int_0^1\Bigl(2tW(W'-W)+t^2(W'-W)^2\Bigr)dt\Bigr]\\
  &=\frac{3}{2\lambda}\E\bigl[W(W'-W)^3\bigr]+\frac{1}{2\lambda}\E\bigl[(W'-W)^4\bigr]\,.
\end{align*}
By exchangeability we obtain that 
\begin{equation*}
 \E\bigl[(W'-W)^4\bigr]=\E\bigl[W'(W'-W)^3\bigr]-\E\bigl[W(W'-W)^3\bigr]=-2\E\bigl[W(W'-W)^3\bigr]
\end{equation*}
yielding the claim. In order to prove \eqref{3mformula} we apply Lemma \ref{exlemma1} with $g(x)=x^2$ leading to the remainder term
\begin{align*}
 E&=\frac{1}{\lambda}\E\Bigl[(W'-W)^2\int_0^1t(W'-W)dt\Bigr] %&=\frac{1}{\lambda}\E\Bigl[W(W'-W)^2\Bigr]+
 =\frac{1}{2\lambda}\E\bigl[(W'-W)^3\bigr]=0\\
 %&=2\E\Bigl[W\frac{1}{2\lambda}\E\bigl[(W'-W)^2\,\bigl|\,\G\bigr]\Bigr]
\end{align*}
again by exchangeability.
\end{proof}

Now, using \eqref{4mformula} we obtain
\begin{align}\label{4mform}
 \E[W^4]&=3\E\Bigl[W^2\frac{n}{2d}\E\bigl[(W'-W)^2\,\bigl|\,X\bigr]\Bigr]-\frac{n}{4d}\E\bigl[(W'-W)^4\bigr]\notag\\%+\frac{3n}{2d}\E\bigl[W(W'-W)^3\bigr]+\frac{n}{2d}\E\bigl[(W'-W)^4\bigr]\notag\\
 &=3\sum_{\substack{M,N\subseteq[n]:\\\abs{M},\abs{N}\leq 2d}}\Bigl(1-\frac{\abs{M}}{2d}\Bigr)\E\bigl[U_N U_M\bigr]-\frac{n}{4d}\E\bigl[(W'-W)^4\bigr]\notag\\%+\frac{3n}{2d}\E\bigl[W(W'-W)^3\bigr]+\frac{n}{2d}\E\bigl[(W'-W)^4\bigr]\notag\\
 &=12\nu^2+3\sum_{\substack{M\subseteq[n]:\\1\leq\abs{M}\leq 2d-1}}\Bigl(1-\frac{\abs{M}}{2d}\Bigr)\Var(U_M)-\frac{n}{4d}\E\bigl[(W'-W)^4\bigr]\,.%\notag\\
 \end{align}
where we have used that $U_\emptyset=\E[W^2]=2\nu$. We have 
\begin{align*}
 \Var(S_2)&=\sum_{J\in\D_d}\Var\bigl(U_J-4W_J\bigr)=\sum_{J\in\D_d}\bigl(\Var(U_J)+16\Var(W_J)-8\Cov(U_J,W_J)\bigr)\\
 &=\sum_{J\in\D_d}\Var(U_J)+32\nu-8\sum_{J\in\D_d}\Cov(U_J,W_J)\notag\\
 &=\sum_{J\in\D_d}\Var(U_J)+32\nu-8\E[W^3]\,,
\end{align*}
where the last equality holds by virtue of \eqref{3mform}. Hence, we have
\begin{align}\label{vars}
\Var(S)&=\Var(S_1)+\frac{1}{4}\Var(S_2)\notag\\
&=\sum_{\substack{M\subseteq[n]:\\1\leq\abs{M}\leq 2d-1\,,\\\abs{M}\not=d}}\Bigl(1-\frac{\abs{M}}{2d}\Bigr)^2\Var(U_M)+\frac14\sum_{J\in\D_d}\Var(U_J)
+8\nu-2\E[W^3]\,.
\end{align}

From \eqref{4mform} and \eqref{vars}, using 
\[\Bigl(1-\frac{\abs{M}}{2d}\Bigr)^2\leq\Bigl(1-\frac{\abs{M}}{2d}\Bigr)\quad\text{for all }M\subseteq[n]\text{ such that }\abs{M}\leq 2d\,,\]
we see that 
\begin{align*}
 &\E[W^4]-12\E[W^3]-12\nu^2+48\nu\\
 &=3\sum_{\substack{M\subseteq[n]:\\1\leq\abs{M}\leq 2d-1}}\Bigl(1-\frac{\abs{M}}{2d}\Bigr)\Var(U_M)-12\E[W^3]+48\nu-\frac{n}{4d}\E\bigl[(W'-W)^4\bigr]\\
 &=3\sum_{\substack{M\subseteq[n]:\\1\leq\abs{M}\leq 2d-1\,,\\\abs{M}\not=d}}\Bigl(1-\frac{\abs{M}}{2d}\Bigr)\Var(U_M)+\frac32\sum_{J\in\D_d}\Var(U_J)-12\E[W^3]+48\nu\\
 &\quad-\frac{n}{4d}\E\bigl[(W'-W)^4\bigr]\\
&\geq 3\biggl(\Var(S_1)+ \frac{1}{2}\Bigl( \sum_{J\in\D_d}\Var(U_J)-8\E[W^3]+32\nu\Bigr)\biggr)-\frac{n}{4d}\E\bigl[(W'-W)^4\bigr]\\
&=3\Bigl(\Var(S_1)+\frac{1}{2}\Var(S_2)\Bigr)-\frac{n}{4d}\E\bigl[(W'-W)^4\bigr]\\
&\geq 3\Var(S)-\frac{n}{4d}\E\bigl[(W'-W)^4\bigr]\,.
\end{align*}
Hence, we obtain that
\begin{align}\label{bound11}
 \Var(S)&\leq \frac{1}{3}\Bigl(\E[W^4]-12\E[W^3]-12\nu^2+48\nu\Bigr) +\frac{n}{12d}\E\bigl[(W'-W)^4\bigr]
 \end{align}
and it thus remains to find a bound on 
\[\frac{n}{d}\E\bigl[(W'-W)^4\bigr]\,.\]

From the definition of the coupling $(W,W')$ and by the inequality $(a+b)^4\leq 8(a^4+b^4)$ we conclude that  
\begin{align}\label{t22}
 \E\babs{W'-W}^4&=\E\Babs{\sum_{J\in\D_d:\alpha\in J}\bigl(W_J^{(\alpha)}-W_J\bigr)}^4=\frac{1}{n}\sum_{j=1}^n\E\biggl[\Bigl(\sum_{J\in\D_d:j\in J}\bigl(W_J^{(j)}-W_J\bigr)\Bigr)^4\biggr]\notag\\
 &\leq\frac{8}{n}\sum_{j=1}^n\E\biggl[\Bigl(\sum_{J\in\D_d:j\in J}W_J^{(j)}\Bigr)^4+\Bigl(\sum_{J\in\D_d:j\in J}W_J\Bigr)^4\biggr]\notag\\
 &=\frac{16}{n}\sum_{j=1}^n\E\biggl[\Bigl(\sum_{J\in\D_d:j\in J}W_J\Bigr)^4\biggr]
 =\frac{16}{n}\sum_{j=1}^n\sum_{\substack{J,K,L,M\in\D_d:\\ j\in J\cap K\cap L\cap M}}\E\bigl[W_JW_KW_LW_M\bigr]\notag\\
 &=\frac{16}{n}\sum_{\substack{(J,K,L,M)\in\D_d^4:\\ J\cap K\cap L\cap M\not=\emptyset}}\abs{J\cap K\cap L\cap M}\;\E\bigl[W_JW_KW_LW_M\bigr]\,.
 \end{align}
 Here, we have used the fact that the sums
 \begin{equation*}
  \sum_{J\in\D_d:j\in J}W_J^{(j)}\quad\text{and}\quad \sum_{J\in\D_d:j\in J}W_J
 \end{equation*}
are identically distributed for each $j\in[n]$. Now, by the definition of $D$ and by the generalized H\"older inequality, for each $(J,K,L,M)\in\D_d^4$ we have 
\begin{align*}
 \babs{\E\bigl[W_JW_KW_LW_M\bigr]}&\leq \Bigl(\E\bigl[W_J^4\bigr]\E\bigl[W_K^4\bigr]\E\bigl[W_L^4\bigr]\E\bigl[W_M^4\bigr]\Bigr)^{1/4}\notag\\
 &\leq\Bigl(D\sigma_J^4D\sigma_K^4D\sigma_L^4D\sigma_M^4\Bigr)^{1/4}\leq D\sigma_J\sigma_K\sigma_L\sigma_M\,.
\end{align*}
Proposition 2.9 of \cite{DP16} implies that 
\begin{equation*}
 \sum_{\substack{(J,K,L,M)\in\D_d^4:\\ J\cap K\cap L\cap M\not=\emptyset}}\sigma_J\sigma_K\sigma_L\sigma_M\leq C_d\rho_n^2\,,
\end{equation*}
where the finite constant $C_d$ only depends on $d$.
Thus, from \eqref{t22} we conclude that 
\begin{align}\label{term2bound}
 \E\babs{W'-W}^4&\leq\frac{16}{n} C_d D_n\rho_n^2\,. 
\end{align}
From \eqref{bound11} and \eqref{term2bound} we have 
\begin{align}\label{bound12}
 \Var(S)&\leq \frac{1}{3}\Babs{\E[W^4]-12\E[W^3]-12\nu^2+48\nu}+\frac{4}{3d}C_dD_n\rho_n^2\,.
\end{align}
Also, from the fact that 
\begin{equation*}
 \E\bigl[(W'-W)^2\bigr]=2\lambda\E[W^2]=\frac{4d\nu}{n}
\end{equation*}
and using the Cauchy-Schwarz inequality we obtain 
\begin{align}\label{t21}
 \frac{1}{6\lambda}\E\babs{W'-W}^3&\leq \frac{n}{6d}\Bigl(\E\bigl[(W'-W)^2\bigr]\Bigr)^{1/2}\Bigl(\E\babs{W'-W}^4\Bigr)^{1/2}\notag\\
 &=\frac{\sqrt{\nu}}{3\sqrt{d}}\Bigl(n\E\babs{W'-W}^4\Bigr)^{1/2}\leq \frac{\sqrt{\nu}}{3\sqrt{d}}\sqrt{16 C_d D_n\rho_n^2}\notag\\
&=  \frac{4\sqrt{\nu}}{3\sqrt{d}}\sqrt{C_d D_n\rho_n^2}\,,
\end{align}
where we have used \eqref{term2bound} again.
Theorem \ref{maintheo} now follows from \eqref{pluginbound}, \eqref{bound12} and \eqref{t21}.

\section{Proof of Theorem \ref{poitheo}}\label{poisson}
For the sake of completeness, we will discuss some further details concerning stochastic analysis  for functionals of a Poisson measure. Throughout the section, we work in the same framework as the one outlined in Section \ref{ss:poissintro}.

\smallskip

For an integer $p\geq1$ we denote by $L^2(\mu^p)$ the Hilbert space of all square-integrable real-valued functions on $\mathcal{Z}^p$ and we write $L^2_s(\mu^p)$ 
for the subspace of those functions in $L^2(\mu^p)$ which are $\mu^p$-a.e. symmetric. Moreover, for ease of notation, we denote by $\norm{\cdot}$ and $\langle \cdot,\cdot\rangle$ the usual norm and scalar product 
on $L^2(\mu^p)$ for whatever value of $p$. We further define $L^2(\mu^0):=\R$. For $f\in L^2(\mu^p)$ we denote by $I_p(f)$ the \textit{multiple Wiener-It\^o integral} of $f$ with respect to $\hat{\eta}$. If $p=0$, then, by convention, 
$I_0(c):=c$ for each $c\in\R$.
The following properties of multiple integrals are standard for all $p,q\geq0$:
\begin{enumerate}[1)]
 \item $I_p(f)=I_p(\tilde{f})$, where $\tilde{f}$ denotes the \textit{canonical symmetrization} of $f\in L^2(\mu^p)$.
 \item $I_p(f)\in L^2(\Prob)$.
 \item $\E\bigl[I_p(f)I_q(g)\bigr]= \delta_{p,q}\,p!\,\langle \tilde{f},\tilde{g}\rangle $, where $\delta_{p,q}$ denotes \textit{Kronecker's delta symbol}.
\end{enumerate}
For $p\geq0$, the Hilbert space consisting of all random variables $I_p(f)$, $f\in L^2(\mu^p)$, is called the \textit{$p$-th Wiener chaos} associated with $\eta$. It is a crucial fact that every $F\in L^2(\Prob)$ admits 
a unique representation 
\begin{equation}\label{chaosdec}
 F=\E[F]+\sum_{p=1}^\infty I_p(f_p)\,,
\end{equation}
where $f_p\in L_s^2(\mu^p)$, $p\geq1$, are suitable symmetric kernel functions. Identity \eqref{chaosdec} is called the \textit{chaotic decomposition} of the functional $F\in L^2(\Prob)$. Next, we briefly introduce the necessary Malliavin operators. The \textit{domain} $\dom D$ of the Malliavin derivative operator $D$ is the set of all $F\in L^2(\Prob)$ such that the chaotic decomposition \eqref{chaosdec} of $F$ satisfies $\sum_{p=1}^\infty p\,p!\norm{f_p}^2<\infty$. For such an $F$ the random function $\mathcal{Z}\ni z\mapsto D_zF\in L^2(\Prob)$ is defined via
\begin{equation*}
 D_zF=\sum_{p=1}^\infty p I_{p-1}\bigl(f_p(z,\cdot)\bigr)\,,
\end{equation*}
where $f_p(z,\cdot)$ is an a.e. symmetric function on $\mathcal{Z}^{p-1}$. Hence, $DF=(D_zF)_{z\in\mathcal{Z}}$ can be viewed as an element of $L^2\bigl(\Omega\times \mathcal{Z},\F\otimes\mathscr{Z},\Prob\otimes \mu\bigr)$. 
Note that, as $\F=\sigma(\eta)$, each $F\in L^2(\Prob)$ can be written as $F=g(\eta)$ for some measurable functional $g$. Then, for $z\in\mathcal{Z}$ we write $F_z:=g(\eta+\delta_z)$. If, furthermore, $F$ happens to be in $\dom D$, then 
it is known that for $\mu$-almost every $z\in\mathcal{Z}$ we have the important formula 
\begin{equation}\label{altDz}
D_zF=F_z-F\,. 
\end{equation}
The domain $\dom L$ of the \textit{Ornstein-Uhlenbeck generator} $L$ is the set of those $F\in L^2(\Prob)$ whose chaotic decomposition \eqref{chaosdec} verifies $\sum_{p=1}^\infty p^2\,p!\norm{f_p}^2<\infty$ and, for $F\in\dom F$, one defines
\begin{equation*}
 LF=-\sum_{p=1}^\infty p I_p(f_p)\,.
\end{equation*}
By definition, $\E[LF]=0$. The domain $\dom L^{-1}$ of the \textit{pseudo-inverse} $L^{-1}$ of $L$ is the class of mean zero elements $F$ of $L^2(\Prob)$. If $F= \sum_{p=1}^\infty I_p(f_p)$ is the chaotic decomposition of such an $F$, then it 
is defined via 
\begin{equation*}
 L^{-1} F=\sum_{p=1}^\infty\frac{1}{p}I_p(f)\,.
\end{equation*}
Note that these definitions imply that 
\begin{equation*}
 LL^{-1}F=F\quad\text{for all }F\in \dom L^{-1}\quad\text{and}\quad L^{-1} LF=F-\E[F]\quad\text{for all } F\in\dom L\,.
\end{equation*}
Finally, we review the definition \textit{Skohorod integral operator} $\delta$. 
Note that for each $u\in L^2\bigl(\Omega\times \mathcal{Z},\F\otimes\mathscr{Z},\Prob\otimes \mu\bigr)$ and each fixed $z\in\mathcal{Z}$ we have a chaotic decomposition of the type
\begin{equation}\label{uz}
 u_z=\sum_{p=0}^\infty I_p\bigl(f_p(z,\cdot)\bigr)
\end{equation}
and, for $p\geq0$, the kernel $f_p(z,\cdot)$ is an element of $L_s^2(\mu^p)$.  Then, the domain $\dom\delta$ of $\delta$ consists of those such $u\in L^2\bigl(\Omega\times \mathcal{Z},\F\otimes\mathscr{Z},\Prob\otimes \mu\bigr)$ whose kernels given by \eqref{uz} satisfy 
\begin{equation*}
 \sum_{p=0}^\infty (p+1)! \norm{f_p}_{L^2(\mu^{p+1})}^2<\infty
\end{equation*}
and, for $u\in\dom\delta$, one lets
\begin{equation*}
 \delta(u)=\sum_{p=0}^\infty I_{p+1}(f_p)\,.
\end{equation*}
The following two identities are essential for the Malliavin-Stein method on the Poisson space. The first one, the \textit{integration by parts formula}, characterizes $\delta$ as the adjoint operator of $D$:
\begin{align}
 \E\bigl[G\delta(u)\bigr]&=\E\bigl[\langle DG,u\rangle_{L^2(\mu)}\bigr]\quad\text{for all } G\in\dom D,\, u\in\dom\delta\,.\label{intparts}\\
 \delta DF&=-LF\quad\text{for all }F\in \dom L\label{commrel}.
\end{align}
We are now in the position to prove our new bounds on the Gamma approximation for functionals on the Poisson space.

\smallskip

\noindent{\bf Proof of Theorem \ref{poitheo}}.
 The proof follows the lines of the proof of Theorem 3.1 of \cite{PSTU} very closely. Fix $h\in\mathcal{H}_2$ and write $f=f_h$ for the solution to the Stein equation \eqref{steineq2} from Theorem \ref{bounds2}. Using the fact that 
 $\E[F]=0$ as well as \eqref{intparts} and \eqref{commrel} we have 
 \begin{align}\label{pp1}
  \E\bigl[Ff(F)\bigr]&=\E\bigl[\bigl(LL^{-1}F\bigr) f(F)\bigr]=\E\bigl[-\delta\bigl(DL^{-1}F\bigr)f(F)\bigr]\notag\\
  &=\E\bigl[\langle Df(F),-DL^{-1}F\rangle_{L^2(\mu)}\bigr]\,.
 \end{align}
Now, for fixed $z\in\mathcal{Z}$, using \eqref{altDz} as well as Taylor's formula, we have 
\begin{align}\label{pp2}
 D_zf(F)&=\bigl(f(Z)\bigr)_z-f(F)=f(F_z)-f(F)\notag\\
 &=f'(F)(F_z-F)+R(F_z-F)=f'(F)(D_zF)+R(D_zF)\,,
\end{align}
where $y\mapsto R(y)$ is a function which satisfies 
\[\abs{R(y)}\leq\frac{\fnorm{f''}}{2}y^2\leq\frac{\max\bigl(1,\frac{2}{\nu}\bigr)\fnorm{h'}+\fnorm{h''}}{2}y^2\leq\max\Bigl(1,\frac{1}{\nu}+\frac12\Bigr) y^2\]
by Theorem \ref{bounds2} (b). Hence, from \eqref{pp1} and \eqref{pp2} we conclude that
\begin{align*}%\label{pp3}
\E\bigl[Ff(F)\bigr]&=\E\bigl[f'(F)\langle DF,-DL^{-1}F\rangle_{L^2(\mu)}\bigr]
+\E\bigl[\langle R(DF),-DL^{-1}F\rangle_{L^2(\mu)}\bigr]
\end{align*}
yielding
\begin{align*}%\label{pp4}
&\babs{\E[h(F)]-\E[h(Z_\nu)]}=\babs{\E\bigl[2(F+\nu)f'(F)-Ff(F)\bigr]}\notag\\
&\leq\babs{\E\bigl[f'(F)\bigl(2(F+\nu)-\langle DF,-DL^{-1}F\rangle_{L^2(\mu)}\bigr]}
+\babs{\E\bigl[\langle R(DF),-DL^{-1}F\rangle_{L^2(\mu)}\bigr]}\notag\\
&\leq \max\Bigl(1,\frac{2}{\nu}\Bigr)\E\babs{2(F+\nu)-\langle DF,-DL^{-1}F\rangle_{L^2(\mu)}}
+\int_{\mathcal{Z}}\E\babs{R(D_zF) D_zL^{-1}F}\mu(dz)\notag\\
&\leq\max\Bigl(1,\frac{2}{\nu}\Bigr)\E\babs{2(F+\nu)-\langle DF,-DL^{-1}F\rangle_{L^2(\mu)}}\\
&\;+\max\Bigl(1,\frac{1}{\nu}+\frac12\Bigr)\int_{\mathcal{Z}}\E\babs{D_zF}^{2}\babs{D_zL^{-1}F}\mu(dz)\,,
\end{align*}
which in turn gives \eqref{genpb1}. Applying Cauchy-Schwarz on \eqref{genpb1} gives \eqref{genpb2}. 
The bounds \eqref{spb1} and \eqref{spb2} easily follow from these by the definitions of the Malliavin operators.\qed

%\begin{remark}\label{radrem}
%In the paper \cite{NPR-ejp} the authors combined Stein's method and Mallivin calculus to derive bounds on the normal approximation of functionals of a sequence of i.i.d. 
%Rademacher random variables. We finish this section by remarking that, using our new bounds on the solution of the Gamma Stein equation, similarly to the derivation of Theorem \ref{poitheo}, we would be able to provide analogous bounds on the Gamma approximation of such Rademacher functionals in the $d_2$-distance. 
%\end{remark}

\section{Proofs of Theorem \ref{t:gaussgamma} and of Proposition \ref{p:gaussgamma} }\label{gauss}

\noindent{\bf Proof of Theorem \ref{t:gaussgamma}.} We have to show that, for every 1-Lpischitz test function $h$, the quantity $ | \E[h(F)] - \E[h(Z_\nu)] |$ is bounded by the right-hand side of \eqref{e:gaussgamma}. We start by assuming that $h$ is twice continuously differentiable and such that $\|h'\|_\infty\leq 1$. Then, we can use Theorem \ref{bounds2} to deduce that that there exists a solution $f_h$ to \eqref{steineq2} such that $f_h$ is continuously differentiable, and $\|f'_h\|_\infty  \leq\max\bigl(1,\frac{2}{\nu}\bigr) \| h'\|_\infty $. It follows that
$$
 | \E[h(F)] - \E[h(Z_\nu)] | = \Babs{\E\Bigl[f_h'(F)\, \E\big \{ 2(F+\nu)-\langle DF,-DL^{-1}F\rangle_{\mathscr{H}} \, \big | \, F \big\} \Big]},
$$
where we have applied the standard integration by parts formula
$$
\E[F f_h(F)] = \E[f'_h(F)\langle DF, -DL^{-1} F\rangle_{\mathscr{H}}],
$$
as well as the definition of conditional expectation. Observe that, in view of the smoothness of $f_h$, such an integration by parts relation holds for any $F\in \mathbb{D}^{1,2}$, irrespective of the fact that $F$ has a density. To deal with a general 1-Lipschitz function $h$, one simply observes that there exists a family $\{h_\epsilon : \epsilon>0 \}$ of functions of class $C^2$ such that: (1) for each $\epsilon$ the first and second derivatives of $h_\epsilon$ are bounded, (2) $\|h'_\epsilon\|_\infty \leq \|h'\|_\infty$, and (3) $\| h-h_\epsilon\|_\infty \to 0$, as $\epsilon \to 0$ (one can take for example $h_\epsilon(x) = \E[h(x+\epsilon N)]$, where $N$ is a standard normal random variable).\qed

\medskip
 
\noindent{\bf Proof of Proposition \ref{p:gaussgamma}.} The fact that \eqref{e:sar} implies that $F_n$ converges in distribution to $Z_\nu$ follows from Theorem \ref{t:gaussgamma}, whereas the estimate \eqref{e:sar1} is an immediate consequence of \cite[Theorem 3.1]{NPo-spa} and of the fact that the Fortet-Mourier distance is bounded (by definition) by $d_1$. To conclude, we have to show that, if $F_n$ converges in distribution to $Z_\nu$, then \eqref{e:sar} is necessarily verified. In order to do that, one can reason exactly as in the proof of \cite[Theorem 3]{APP} and deduce that, if $F_n$ converges in distribution to $Z_\nu$, then, as $n\to \infty$ and for every fixed $M\in (0, \infty)$,
$$
\sup_{\phi\in \mathscr{F}_M} \E\left[ \phi(F_n) (2(F_n+\nu)-\langle DF_n,-DL^{-1}F_n\rangle_{\mathscr{H}} )\right] \to 0,
$$
where $\mathscr{F}_M$ denotes the class of all Borel functions that are bounded by 1, and with support contained in $[-M,M]$. It follows that 
\begin{eqnarray*}
&&\E\left[ \left|  \E \left\{ 2(F_n+\nu)-\langle DF_n,-DL^{-1}F_n\rangle_{\mathscr{H}}\, \big |\, F_n \right\} \right|\right] \\
&& = \sup_{\|\phi\|_\infty\leq 1} \E\left[ \phi(F_n)  \E \left\{ 2(F_n+\nu)-\langle DF_n,-DL^{-1}F_n\rangle_{\mathscr{H}}  \, \big |\, F_n \right\} \right] \\
&& \leq \sup_{\phi\in  \mathscr{F}_M} \E\left[ \phi(F_n)\,  (2(F_n+\nu)-\langle DF_n,-DL^{-1}F_n\rangle_{\mathscr{H}} )\right]\\
&& \quad\quad + \sqrt{\Prob[|F_n| >M]} \times \sup_k \E\left[ (2(F_k+\nu)-\langle DF_k,-DL^{-1}F_k\rangle_{\mathscr{H}} )^2\right]^{1/2},
\end{eqnarray*}
and the conclusion is obtained by first letting $n\to \infty$, and next letting $M\to \infty$, where one has to use the fact that
$$
\sup_k \E\left[ (2(F_k+\nu)-\langle DF_k,-DL^{-1}F_k\rangle_{\mathscr{H}} )^2\right]^{1/2}<\infty,
$$
by virtue of the usual hypercontractivity properties enjoyed by random variables living in a finite sum of Wiener chaoses --- see e.g. \cite[Corollary 2.8.15]{NouPecbook}.
\qed

{
\section{Proof of Lemma \ref{smoothing}}\label{appendix} 
The proof { refines findings from the unpublished PhD dissertation} \cite{Doe12c}. For $\rho>0$ denote by 
\begin{equation*}
 k_\rho(x):=\frac{\rho}{\sqrt{2\pi}}e^{-x^2\rho^2/2}=\rho\phi(\rho x)\,,\quad x\in\R\,,
\end{equation*}
the density of the centered normal distribution with variance $\rho^{-2}$, which we use as a mollifier. For a Lipschitz-continuous function $h$ on $\R$, denote by $h_\rho:=h\ast k_\rho=k_\rho\ast h$ the convolution of 
$h$ and $k_\rho$, given by 
\begin{equation*}
 h_\rho(x)=(h\ast k_\rho)(x)=\int_\R h(y)k_\rho(x-y)dy=\int_\R k_\rho(y)h(x-y)dy\,,\quad x\in\R\,.
\end{equation*}
Note that, according to Rademacher's theorem, $h$ is Lebesgue-a.e. differentiable with a bounded derivative. 
In what follows we denote by $h'$ an arbitrary bounded and measurable version of its derivative. 

\begin{prop}\label{moll}
Fix $\rho>0$. For any Lipschitz-continuous function $h:\R\rightarrow\R$, the function $h_\rho$ is in $C^\infty(\R)$ and for each integer $m\geq1$ we have:
\begin{enumerate}[{\normalfont (a)}]
 \item $h_\rho^{(m)}=h\ast k_\rho^{(m)}$
 \item $h_\rho^{(m)}=h'\ast k_\rho^{(m-1)}$
\end{enumerate}
\end{prop}

\begin{proof}
We only prove (b) for $m=1$ because (a) and the remaining part of (b) are standard facts about the differentiation of mollified functions, as is the fact that $h_\rho\in C^\infty(\R)$. 
Fix $x\in\R$. Then, for almost every $y\in\R$, the function 
\begin{equation*}
 D_{x,u}h(y):=\frac{h(x+u-y)-h(x-y)}{u}-h'(x-y)\,, \quad u\not=0\,,
\end{equation*}
converges to $0$ as $u\to0$. Furthermore, we have 
\begin{equation*}
 \babs{D_{x,u}h(y)}\leq 2\fnorm{h'}
\end{equation*}
for all $x,y,u$, where $\fnorm{h'}$ is the minimal Lipschitz constant for $h$. Hence, using dominated convergence, we conclude
\begin{align*}
 \Babs{\frac{h_\rho(x+u)-h\rho(x)}{u}-(h'\ast k_\rho) (x)}&=\Babs{\int_\R k_\rho(y)D_{x,u}h(y)dy}\leq \int_\R k_\rho(y)\babs{D_{x,u}h(y)}dy\notag\\
 &\longrightarrow 0\,,\quad\text{as }u\to0\,. 
\end{align*}
This implies that $h_\rho$ is differentiable at $x$ and $h_\rho'(x)=(h'\ast k_\rho)(x)$.
\end{proof}

\begin{cor}\label{moll2}
For each integer $m\geq1$ we have 
\[\fnorm{h_\rho^{(m)}}\leq\fnorm{h'}\int_\R\abs{k_\rho^{(m-1)}(y)}dy\leq C_m\rho^{m-1} \fnorm{h'}\,,\]
where the finite constant $C_m>0$ is defined by 
\[C_m:=\int_\R\abs{H_{m-1}(x)}\phi(x)dx\,,\]
and, for $j\in\N_0$, $H_j$ denotes the $j$-th monic Hermite polynomial. In particular, we have $C_1=1$ and $C_2=\sqrt{\frac{2}{\pi}}$ { and $C^2_m \leq {(m-1)!}$ for every $m\geq 3$.}
\end{cor}

\begin{proof}
 It is well-known (by the Rodrigues formula) that we have 
\begin{equation*}
 \phi^{(m)}(x)=(-1)^m H_m(x)\phi(x)\,,\quad x\in\R\,.
\end{equation*}
Hence, since $k_\rho(x)=\rho \phi(\rho x)$, for each $j\in\N_0$, 
\begin{equation*}
 k_\rho^{(j)}(x)=\rho^{j+1}(-1)^{j} H_j(\rho x)\phi(\rho x)=\rho^j \rho (-1)^{j} H_j(\rho x)\phi(\rho x)\,,\quad x\in\R\,.
\end{equation*}
Thus, by Proposition \ref{moll} we conclude that, for $m\in\N$ and $x\in\R$, we have
\begin{align*}
\abs{h_\rho^{(m)}(x)}&=\Bigl|\int_\R h'(y)k_\rho^{(m-1)}(x-y)dy\Bigr|\leq\fnorm{h'}\int_\R \abs{k_\rho^{(m-1)}(u)}du\notag\\
&\leq\fnorm{h'} \rho^{m-1}\int_\R \rho \babs{H_{m-1}(\rho u)}\phi(\rho u)du\notag\\
&=\rho^{m-1}\fnorm{h'} \int_\R\babs{H_{m-1}(y)}\phi(y)dy=C_m\rho^{m-1} \fnorm{h'}\,.
\end{align*}
\end{proof}

\begin{prop}\label{moll3}
 For each $\rho>0$ and each Lipschitz-continuous function $h$ we have 
 \[\fnorm{h-h_\rho}\leq\frac{\fnorm{h'}}{\rho}\sqrt{\frac{2}{\pi}}\,.\]
\end{prop}

\begin{proof}
 Fix $x\in\R$. Then, we have 
 \begin{align*}
  \babs{h_\rho(x)-h(x)}&=\Babs{\int_\R k_\rho(y)\bigl(h(x-y)-h(x)\bigr)dy}
  \leq\fnorm{h'}\int_\R\abs{y}k_\rho(y)dy\\
  &=\frac{\fnorm{h'}}{\rho}\int_\R\abs{u}\phi(u)du=\frac{\fnorm{h'}}{\rho}\sqrt{\frac{2}{\pi}}\,,
 \end{align*}
as claimed.
\end{proof}

\begin{proof}[End of the proof of Lemma \ref{smoothing}]
We may assume that $d_2(X,Y)>0$ because otherwise $X$ and $Y$ have the same distribution and $d_1(X,Y)=0$ as well.
Let $h$ be a Lipschitz-continuous function with Lipschitz-constant $1$. Note that by Corollary \ref{moll2} , for $\rho\geq\sqrt{\frac{\pi}{2}}$, we have $a_\rho h_\rho\in \mathcal{H}_2$, where $a_\rho:=\sqrt{\frac{\pi}{2}}\rho^{-1}$.
Hence, using Proposition \ref{moll3}, for $\rho\geq\sqrt{\frac{\pi}{2}}$ we obtain
\begin{align}\label{sl1}
 \babs{\E[h(X)]-\E[h(Y)]}&\leq \babs{\E[h(X)]-\E[h_\rho(X)]}\notag\\
 &\;+\babs{\E[h_\rho(X)]-\E[h_\rho(Y)]}
 +\babs{\E[h_\rho(Y)]-\E[h(Y)]}\notag\\
 &\leq\;2\fnorm{h-h_\rho}+\babs{\E[h_\rho(X)]-\E[h_\rho(Y)]}\notag\\
 &=2\fnorm{h-h_\rho}+a_\rho^{-1}\babs{\E[(a_\rho h_\rho)(X)]-\E[(a_\rho h_\rho)(Y)]}\notag\\
& \leq \frac{2^{3/2}}{\rho\sqrt{\pi}}+a_\rho^{-1}d_2(X,Y)\notag\\
&=\frac{2^{3/2}}{\rho\sqrt{\pi}}+ \rho\sqrt{\frac{2}{\pi}}d_2(X,Y)=:g(\rho)\,.
\end{align}
It can be checked that $g$ assumes its minimum on $(0,\infty)$ at 
\[\rho_0:=\frac{\sqrt{2}}{\sqrt{d_2(X,Y)}}\geq\sqrt{\frac{\pi}{2}}\]
because $d_2(X,Y)\leq1$ by assumption. Thus, as \eqref{sl1} holds uniformly over all $1$-Lipschitz functions $h$ and all $\rho\geq\sqrt{\frac{\pi}{2}}$, we obtain
\begin{equation*}
d_1(X,Y)\leq g(\rho_0)= \frac{4}{\sqrt{\pi}}\sqrt{d_2(X,Y)}\,.
\end{equation*}

\end{proof}

}

\normalem
\bibliography{gammafinal}{}
\bibliographystyle{alpha}
\end{document}